\documentclass[letterpaper,10pt]{scrartcl}
\usepackage{fullpage}

\usepackage[utf8]{inputenc}
\usepackage[english]{babel}
\usepackage{amsmath,amstext,amsbsy,amsopn,amscd,amsxtra,upref,amssymb}
\usepackage{amsfonts,bm}
\usepackage{amsthm}
\usepackage{color}
\usepackage{graphicx}
\usepackage{sidecap}
\usepackage{multirow}
\usepackage{booktabs}
\usepackage{bm}
\usepackage{fullpage}

\usepackage{romanbar}

\usepackage{tabularx}

\usepackage[font=small]{caption}
\usepackage{subfig}

\usepackage{longtable}
\usepackage{rotating}

\usepackage[ruled, linesnumbered]{algorithm2e}
\SetKwComment{Comment}{/* }{ */}
\extrafloats{100}

\usepackage{lineno,hyperref}

\usepackage{amsmath,amstext,amsbsy,amsopn,amscd,amsxtra,upref,amssymb}
\usepackage{amsfonts,bm}
\usepackage{amsthm}
\graphicspath{{figures/}, {figures/snapshots/}}
\DeclareGraphicsExtensions{.pdf,.eps,.png,.jpg,.jpeg}

\DeclareOldFontCommand{\bf}{\normalfont\bfseries}{\mathbf}

\newcommand{\KL}[2]{\mathrm{KL}\left( #1 \mid\mid #2 \right)}
\newcommand{\dhell}[2]{\mathrm{d}_{\mathrm{Hell}}\left( #1 , #2 \right)}

\newcommand{\btheta}{{\boldsymbol \theta}}
\newcommand{\bsigma}{{\boldsymbol \sigma}}
\newcommand{\btau}{{\boldsymbol \tau}}
\newcommand{\beps}{{\boldsymbol \epsilon}}
\renewcommand{\beps}{{\boldsymbol \epsilon}}

\newcommand{\bx}{{\mathbf{x}}}
\newcommand{\by}{{\mathbf{y}}}
\newcommand{\bz}{{\mathbf{z}}}
\newcommand{\bg}{{\mathbf{g}}}
\newcommand{\bu}{{\boldsymbol u}}
\newcommand{\eye}{{\mathbf{I}}}

\newcommand{\R}{\mathbb{R}}

\newcommand{\norm}[1]{\left\lVert #1 \right\rVert }
\newcommand{\inner}[2]{\left\langle #1,\ #2 \right \rangle}

\newcommand{\zeros}{{\mathbf 0}}

\newcommand{\dvg}{\nabla \cdot}

\newcommand{\bfx}{\boldsymbol x}

\newcommand{\bfv}{\boldsymbol v}

\newcommand{\bfu}{\boldsymbol u}

\newtheorem{assumption}{Assumption}
\newtheorem{proposition}{Proposition}
\newtheorem{remark}{Remark}

\newenvironment{keywords}%
   {\begin{trivlist}\item[]{\bfseries\sffamily Keywords:}\ }
   {\end{trivlist}}

\ifpdf
\hypersetup{
  pdftitle={Further analysis of multilevel Stein variational gradient descent with an application to the Bayesian inference of glacier ice models}, 
  pdfauthor={Terrence Alsup and Tucker Hartland and Benjamin Peherstorfer and Noemi Petra} 
}
\fi

\begin{document}

\title{Further analysis of multilevel Stein variational gradient descent with an application to the Bayesian inference of glacier ice models}

\date{\today}

\author{Terrence Alsup\footnote{Courant Institute of Mathematical Sciences, New York University, NY, USA} \and Tucker Hartland\footnote{Department of Applied Mathematics, University of California, Merced, CA, USA} \and Benjamin Peherstorfer\footnotemark[1] \and Noemi Petra\footnotemark[2]}

\maketitle

\begin{abstract}
Multilevel Stein variational gradient descent is a method for particle-based variational inference that leverages hierarchies of surrogate target distributions with varying costs and fidelity to computationally speed up inference. The contribution of this work is twofold. First, an extension of a previous cost complexity analysis is presented that applies even when the exponential convergence rate of single-level Stein variational gradient descent depends on iteration-varying parameters. Second, multilevel Stein variational gradient descent is applied to a large-scale Bayesian inverse problem of inferring discretized basal sliding coefficient fields of the Arolla glacier ice. The numerical experiments demonstrate that the multilevel version achieves orders of magnitude speedups compared to its single-level version.
\end{abstract}


\section{Introduction}

Bayesian inference is a ubiquitous and flexible tool for updating a belief (i.e., learning) about a quantity of interest when data are observed, which ultimately can be used to inform downstream decision-making. In particular, Bayesian inverse problems allow one to derive knowledge from data through the lens of physics-based models. These problems can be formulated as follows: given observational data, a physics-based model, and prior information about the model inputs, find a posterior probability distribution for the inputs that reflects the knowledge about the inputs in terms of the observed data and prior. Typically, the physics-based models are given in the form of an input-to-observation map 
that is based on a system of partial differential equations (PDEs). 
The computational task underlying Bayesian inference is approximating posterior probability distributions to compute expectations and to quantify uncertainties. There are multiple ways of computationally exploring posterior distributions to gain insights, reaching from Markov chain Monte Carlo to variational methods~\cite{KS,S,doi:10.1137/19M1247176}. 

In this work, we make use of Stein variational gradient descent (SVGD) \cite{NIPS2016_b3ba8f1b}, which is a method for particle-based variational inference, to approximate posterior distributions. It builds on Stein's identity to formulate an update step for the particles that can be realized numerically in an efficient manner via a reproducing kernel Hilbert space. There are various extensions to SVGD such as exploiting curvature information of the target distribution with a corresponding Newton method \cite{NEURIPS2018_fdaa09fc} as well as using adaptive kernels as in \cite{duncan2019geometry, SteinMatrix}. Specifically for Bayesian inverse problems, SVGD has been extended to take advantage of low-dimensional structure \cite{CUI2016109} in the posterior distribution  \cite{NEURIPS2019_eea5d933} and the model states \cite{doi:10.1137/20M1321589}. Much effort has been put into understanding the convergence and statistical properties of SVGD and its variants. The study of convergence of SVGD was sparked primarily by \cite{NIPS2017_17ed8abe,doi:10.1137/18M1187611}, which showed that in the mean-field limit SVGD follows a gradient flow with respect to the Kullback-Leibler (KL) divergence. Similar results were later shown for the chi-squared divergence in \cite{Chewi20}. Pre-asymptotic convergence results in both the number of samples and the discrete-time setting remains open, but progress in this direction has been made in \cite{Korba}. There also has been work on understanding and improving the performance of SVGD in high dimensions \cite{Ba2019}.

We focus on the multilevel extension of SVGD (MLSVGD), which was introduced in~\cite{MLSVGD} and leverages hierarchies of approximations of a target posterior distribution with varying costs and fidelity to computationally speed up inference. Such approximations can be obtained via, e.g., coarse and fine discretizations of the governing equations of the physics-based models as well as surrogate models \cite{doi:10.1137/16M1082469} and simplified-physics models \cite{KONRAD2022110898}. Multilevel methods have a long tradition in scientific computing and computational statistics.  The MLSVGD approach is motivated by multi-fidelity and multilevel methods such as multilevel and multi-fidelity Monte Carlo \cite{MLMC, doi:10.1287/opre.1070.0496, cliffe_multilevel_2011,PWK16MFMCAsymptotics,Peherstorfer18SciTechAIAA} and Markov chain Monte Carlo (MCMC) methods \cite{doi:10.1137/130915005,DAMLMC,PM18MultiTM}. The MLSVGD also shares similarities with multilevel sequential Monte Carlo \cite{BESKOS20171417, LATZ2018154, doi:10.1137/19M1289601,PKW17MFCE} and importance sampling \cite{AP20Context, doi:10.1137/16M1082469,doi:10.1137/19M1289601}, multilevel particle filters \cite{doi:10.1137/15M1038232}, multilevel preconditioning \cite{doi:10.1137/1.9780898719505,HackbushMG,MLFT}, and multilevel ensemble Kalman methods \cite{doi:10.1137/15M100955X,doi:10.1137/21M1423762}, which all use hierarchies of surrogate models to generate samples sequentially. The work \cite{MLSVGD} provides a cost complexity analysis of MLSVGD that shows speedups compared to single-level SVGD; but it relies on an exponential convergence rate of SVGD with a fixed parameter and thus is limited in scope.

In this work, we contribute an analysis of MLSVGD that applies when the parameter of the convergence rate depends on the MLSVGD iteration. 
The finding is that the same cost complexity is achieved as in the fixed-parameter setting as long as mild conditions on the parameter can be made. We also show how the constants in the cost complexity change and that MLSVGD achieves speedups over single-level SVGD when the constants in the convergence rate of SVGD lead to a slow error decay. This is directly applicable to Bayesian inverse problems, where we show that the assumptions of the cost complexity analysis are satisfied in typical settings.

We numerically demonstrate MLSVGD on a Bayesian inverse problem of inferring a discretized basal sliding coefficient field from velocity observations at the surface of the Arolla glacier~\cite{Noemi2014}; see also \cite{ISMIP-HOM}. The numerical setup builds on FEniCS~\cite{fenics} and hIPPYlib \cite{VillaPetraGhattas16,VillaPetraGhattas18,VillaPetraGhattas21,https://doi.org/10.48550/arxiv.2112.00713}, which allows for fast gradient-based inference via adjoints. The numerical results show that MLSVGD performs inference at a fraction of the cost of inference with SVGD and that it leads to higher quality particles with respect to the maximum mean discrepancy (MMD) \cite{Gretton} than samples obtained with a variant of MCMC.  

The manuscript is organized as follows.  In Section~\ref{sec:Prelim} we outline preliminaries on Bayesian inverse problems, SVGD, and previous work on MLSVGD. Section~\ref{sec:3} introduces extended cost complexity bounds for MLSVGD that apply in more general settings. In Section~\ref{sec:NumExp} we demonstrate improvements by several factors in terms of computational savings of MLSVGD over SVGD for inferring the basal sliding coefficient in the Arolla glacier ice model. We conclude in Section~\ref{sec:Conc}.

\section{Preliminaries}\label{sec:Prelim}

In this section, Bayesian inverse problems are reviewed and it is discussed how they are related to sampling from a target distribution with, e.g., SVGD and MLSVGD.  

\subsection{Bayesian inverse problems}
\label{sec:2p1}

Let $G : \Theta \to \mathbb{R}^{q}$ denote a parameter-to-observable map and consider noisy data $\by = G(\btheta^*) + {\boldsymbol \eta}$, where ${\boldsymbol \eta} \sim N(\zeros, \Gamma)$ with known noise covariance matrix $\Gamma \in \mathbb{R}^{q\times q}$ and $\btheta^* \in \Theta \subset \mathbb{R}^d$.  Given a prior $\pi_0:\Theta \to \mathbb{R}$, the target posterior density is
\begin{equation}\label{eq:Prelim:TargetPDF}
    \pi(\btheta) \propto \exp\left(-\frac{1}{2}\norm{ \by - G(\btheta) }^2_{\Gamma^{-1}} \right) \pi_0(\btheta)\, ,
\end{equation}
where $\norm{\bfv}_{\Gamma^{-1}}^2 = \inner{\bfv}{\Gamma^{-1} \bfv}$.  In many computational science and engineering applications, the parameter-to-observable map $G$ depends on the solution of an underlying system of PDEs, which means that it cannot be evaluated directly. Instead,  one must resort to a numerical method that discretizes the underlying PDE problem to approximately evaluate $G$.  Let $G^{(\ell)}$ be such an approximate parameter-to-observable map, where the index $\ell$ denotes the fidelity and corresponds to,  e.g., the mesh width or number of grid points. The larger $\ell$, the more accurate the approximation $G^{(\ell)}$ of $G$ in the following. The corresponding low-fidelity posterior is
\begin{equation}
    \pi^{(\ell)}(\btheta) \propto \exp\left(-\frac{1}{2}\norm{ \by - G^{(\ell)}(\btheta) }^2_{\Gamma^{-1}} \right) \pi_0(\btheta)\, .
\label{eq:SurrogateDensity}
\end{equation}
Increasing the level $\ell$, gives rise to a sequence of densities $(\pi^{(\ell)})_{\ell = 1}^{\infty}$ that converges pointwise $\pi^{(\ell)}(\btheta) \to \pi(\btheta)$ for every $\btheta \in \Theta$ so that the sequence of random variables $\btheta^{(\ell)}\sim \pi^{(\ell)}$ converges weakly to $\btheta \sim \pi$.  

Our aim is to compute quantities of interest of the form
\begin{equation}
    \mathbb{E}_{\pi}[f] = \int_{\Theta} f(\btheta) \mathrm{d}\pi(\btheta)\, ,
    \label{eq:qoi}
\end{equation}
for given test functions $f : \Theta \to \mathbb{R}$. Because $\pi$ is not readily available, one typically selects a sufficiently accurate $G^{(L)}$ and approximates the quantity of interest with respect to the corresponding density $\pi^{(L)}$, \begin{equation}
    \mathbb{E}_{\pi^{(L)}}[f] = \int_{\Theta} f(\btheta) \mathrm{d}\pi^{(L)}(\btheta)\,.
    \label{eq:qoiL}
\end{equation}
A well-established approach to estimate~\eqref{eq:qoiL} using Monte Carlo involves drawing samples $\btheta^{[1]},\ldots,\btheta^{[N]}$ of the distribution with density $\pi^{(L)}$ and computing 
\begin{equation}
    \hat{f} = \frac{1}{N}\sum_{i=1}^N f(\btheta^{[i]} ) \, .
\label{eq:MonteCarlo}
\end{equation}
For example, the samples $\btheta^{[1]},\ldots,\btheta^{[N]}$ may be i.i.d.~or come from a realization of an ergodic Markov chain. This gives rise to two sources of error with respect to the quantity of interest \eqref{eq:qoi}.  
The first source of error is the Monte Carlo error of estimating the expectation in \eqref{eq:qoiL} with \eqref{eq:MonteCarlo}, while the second source of error is due to using the deterministic approximation $G^{(L)}$ of $G$, and thus $\pi^{(L)}$ instead of $\pi$.   The Monte Carlo error can be controlled with the number of samples $N$. The second error is controlled by the level $L$, which can be selected via, e.g., the Hellinger distance so that 
    \[
        \dhell{\pi^{(L)}}{\pi} \le \epsilon
    \]
    holds for some tolerance $\epsilon > 0$. 
The Hellinger distance is particularly useful because it is a metric on the space of probability measures, allowing to separate the deterministic error due to the fidelity and the statistical error due to sampling, and can be bounded from above by the KL divergence
\begin{equation}
    2\dhell{\mu_1}{\mu_2}^2 \le \KL{\mu_1}{\mu_2} \, .
\label{eq:HellKLBound}
\end{equation}

\subsection{Stein variational gradient descent}
\label{sec:2p2}

We now briefly review SVGD \cite{NIPS2016_b3ba8f1b} that aims to derive a sequence of distributions to minimize the KL divergence with respect to the target density $\pi^{(L)}$.  Once convergence has been reached, the quantity of interest~\eqref{eq:qoi} can be estimated using particles of the distribution. 

Let $\mathcal{H}$ be a reproducing kernel Hilbert space (RKHS) with positive definite kernel $K:\Theta \times \Theta \to \mathbb{R}$ of functions $g:\Theta \to \mathbb{R}$ and let $\mathcal{H}^d \simeq \mathcal{H} \times \cdots \times \mathcal{H}$ be the corresponding RKHS of vector fields $\bg = (g_1,\ldots,g_d):\Theta^d \to \mathbb{R}^d$.  Define the KL functional
\[
    J_{\mu}(\bg) = \KL{ (\eye - \bg)_{\#} \mu  }{ \pi^{(L)} } \, ,
\]
where $(\eye - \bg)_{\#}\mu$ denotes the pushforward measure of $\mu$ under the map $\eye - \bg$, so that if $\btheta \sim \mu$, then $\btheta - \bg(\btheta) \sim (\eye - \bg)_{\#}\mu$, with $\eye : \mathbb{R}^d \to \mathbb{R}^d$ being the identity map and $\bg \in \mathcal{H}^d$.  From the particle point of view, SVGD starts with an initial particle $\btheta_0 \sim \mu_0$ and evolves it according to the gradient dynamics, also known as the mean-field characteristic flow \cite{doi:10.1137/18M1187611},
\begin{equation}
       \dot{\btheta}_t = - \nabla J_{\mu_t}(\zeros)(\btheta_t) \, ,
      \label{eq:SVGDParticle}
\end{equation}
where $\mu_t$ denotes the density of $\btheta_t$ at time $t \ge 0$.  The gradient $\nabla J_{\mu}(\zeros)$ can be computed using the following relation derived in \cite{NIPS2016_b3ba8f1b}
    \begin{equation}
        \nabla J_{\mu}(\zeros)\left( \btheta \right) = 
        -\mathbb{E}_{\bz \sim \mu}\left[  K(\bz, \btheta) \nabla \log \pi^{(L)}(\bz) + \nabla_1 K(\bz, \btheta)  \right] \, ,
        \label{eq:SteinIdentity}
    \end{equation}
where $\nabla_1$ denotes the gradient with respect to the first argument.  The density $\mu_t$ is the solution of the nonlinear Fokker-Planck equation corresponding to the particle evolution \eqref{eq:SVGDParticle}
    \begin{equation}
        \partial_t \mu_t(\btheta) = -\dvg\left(  \mu_t (\btheta) \mathbb{E}_{\bz \sim \mu_t} \left[  K(\bz, \btheta)\nabla \log \pi^{(L)}(\bz) + \nabla_1 K(\bz, \btheta) \right] \right) \, .
        \label{eq:SVGDDensity}
    \end{equation}

Much of the analysis of SVGD revolves around understanding the solution $\mu_t$ to the, potentially high-dimensional, nonlinear PDE \eqref{eq:SVGDDensity}.  One key result that arises due to the gradient flow dynamics \eqref{eq:SVGDDensity} is that the KL divergence $\KL{\mu_t}{\pi^{(L)}}$ converges to zero and it was shown in \cite[Theorem 3.4]{NIPS2017_17ed8abe} that for a solution $\mu_t$ of \eqref{eq:SVGDDensity} with $\KL{\mu_0}{\pi^{(L)}} < \infty$, it holds that
    \begin{equation}
        \frac{\mathrm{d}}{\mathrm{d}t} \KL{\mu_t}{\pi^{(L)}} = -\mathbb{D}(\mu_t \mid \mid \pi^{(L)} )^2\, ,
       \label{eq:LyapunovKL}
    \end{equation}
where
    \[
        \mathbb{D}(\mu \mid \mid \nu) = \underset{\bg \in \mathcal{H}^d}{\max}\ \left\{ \mathbb{E}_{\btheta \sim \mu}[\nabla \log \nu(\btheta)^{\top}\bg(\btheta) + \dvg \bg(\btheta)] : \norm{\bg}_{\mathcal{H}} \le 1 \right\}
    \]
is the Stein discrepancy, guaranteeing that the KL divergence from the target decreases monotonically.  The result~\eqref{eq:LyapunovKL} provides motivation for considering a monotone convergence behavior as in Assumption~\ref{asm:SVGDRateGeneral} later.  The Stein discrepancy $\mathbb{D}(\mu \mid \mid \nu) = 0$ if $\mu = \nu$, but the converse may only be valid if the space $\mathcal{H}$ is sufficiently rich and can otherwise result in a biased estimate of the quantity of interest~\eqref{eq:qoi}.
    
    \begin{remark}
    There is a strong connection between SVGD and the unadjusted Langevin algorithm \cite{JKO98} in the sense that the Langevin algorithm evolves a density that minimizes the KL divergence in the Wasserstein metric as opposed to a SVGD that uses a kernelized Wasserstein metric \cite{Chewi20}.
    \end{remark}

\subsection{Multilevel Stein variational gradient descent}
\label{sec:2p3}

The work \cite{MLSVGD} introduced a multilevel variant of SVGD and showed that one can achieve a cost complexity reduction by integrating the continuous-time mean-field flow \eqref{eq:SVGDParticle} with successively more accurate and more expensive-to-evaluate low-fidelity densities $\pi^{(1)},\ldots,\pi^{(L)}$ as opposed to integrating only with respect to the high-fidelity density $\pi^{(L)}$.  The analysis in \cite{MLSVGD} of the cost complexity relied on the following assumptions.

    \begin{assumption}
    The costs $c_{\ell}$ of integrating \eqref{eq:SVGDDensity} with target density $\pi^{(\ell)}$ for a unit time interval are bounded as
    \[
        c_{\ell} \leq c_0 s^{\gamma \ell}\,,\qquad \ell \in \mathbb{N}\,,
    \]
    with constants $c_0, \gamma > 0$ independent of $\ell$ and $s > 1$.
    \label{asm:Costs}
    \end{assumption}

    \begin{assumption}
    There exists $\alpha, k_0, k_1 > 0$ independent of $\ell$ such that $\KL{\mu_0}{\pi^{(\ell)}} \leq k_0$ for all $\ell \in \mathbb{N}$ and 
    \[
        \KL{\pi^{(\ell)}}{\pi} \leq k_1 s^{-\alpha \ell}\,,\qquad \ell \in \mathbb{N}\,,
    \]
    where $s$ is the same constant independent of $\ell$ as in Assumption~\ref{asm:Costs} and $\mu_0$ is the initial distribution.
    \label{asm:BiasRate}
    \end{assumption}

    \begin{assumption}
    There exists a rate $\lambda > 0$ such that for any initial distribution $\nu_0$
    \begin{equation}\label{eq:KLExpBound}
        \KL{\nu_t}{\pi^{(\ell)}} \leq \mathrm e^{-\lambda t} \KL{\nu_0}{\pi^{(\ell)}}\,,\qquad \ell \in \mathbb{N}\,,
    \end{equation}
    holds, where $\nu_t$ solves the mean-field SVGD equation \eqref{eq:SVGDDensity} at time $t$.
    \label{asm:SVGDRate}
    \end{assumption}

Single-level SVGD derives an approximation $\mu^{\mathrm{SL}}$ such that $\dhell{\mu^{\mathrm{SL}}}{\pi} \le \epsilon$, by selecting a high-fidelity approximation $\pi^{(L)}$ with 
\begin{equation}
\dhell{\pi^{(L)}}{\pi} \le \epsilon/2
\label{eq:Prelim:CondOnL}
\end{equation}
and then integrating~\eqref{eq:SVGDDensity} with respect to $\pi^{(L)}$ for time $T_{\mathrm{SL}}(\epsilon)$
    \begin{equation}
        T_{\mathrm{SL}}(\epsilon) =   \min \left\{ t \ge 0\ :\ \dhell{\mu_t}{\pi^{(L)}} \le \frac{\epsilon}{2} \right\}\, .
    \label{eq:SLTime}
    \end{equation}
This leads to the cost of  
single-level SVGD
    \[
        c_{\mathrm{SL}}(\epsilon) = c_{L(\epsilon)} T_{\mathrm{SL}}(\epsilon) \, ,
    \]
    where the cost $c_{L(\epsilon)}$ depends on $\epsilon$ through the level $L$ that is selected such that \eqref{eq:Prelim:CondOnL} holds. 
    In the remainder of this manuscript, for brevity, we drop the explicit dependence $L = L(\epsilon)$ and similarly $T_{\mathrm{SL}} = T_{\mathrm{SL}}(\epsilon)$ when $\epsilon$ is fixed.
    The following upper bound for the cost complexity of single-level SVGD was derived in \cite[Proposition 2]{MLSVGD}.

    \begin{proposition}
    If Assumptions~\ref{asm:Costs}--\ref{asm:SVGDRate} hold, then the costs of  
    single-level SVGD to obtain $\mu^{\mathrm{SL}}$ with
    \[
        \dhell{\mu^{\mathrm{SL}}}{\pi} \leq \epsilon\, ,
    \]
    is bounded as
    \begin{equation}
        c_{\mathrm{SL}}(\epsilon) \le \frac{2c_0 s^{\gamma}}{\lambda} \left( \frac{\sqrt{2k_1}}{\epsilon}  \right)^{2\gamma/\alpha} \log \left( \frac{\sqrt{\KL{\mu_0}{\pi^{(L)}}}}{\sqrt{2}\epsilon} \right) \, ,
    \label{eq:MLSVGD:SLSVGDCost}
    \end{equation}
    with high-fidelity level
    \begin{equation}
        L = \left\lceil \frac{1}{2 \alpha} \log_s \left( \frac{\sqrt{2 k_1}}{\epsilon} \right) \right \rceil \, .
    \label{eq:LevelEpsilon}
    \end{equation}
    \label{prop:SVGDComplexity}
    \end{proposition}
From~\eqref{eq:MLSVGD:SLSVGDCost}, a higher initial KL divergence $\KL{\mu_0}{\pi^{(L)}}$ or a slower convergence rate (small $\lambda$) for SVGD will result in a larger cost complexity to obtain the single-level SVGD approximation of $\pi^{(L)}$.

    In contrast to single-level SVGD, the MLSVGD method introduced in \cite{MLSVGD}  first integrates with respect to the cheapest and least accurate lowest fidelity density $\pi^{(1)}$ for time $T_1 > 0$ to obtain density $\mu^{(1)}_{T_{1}}$, which serves as an initial density for the next level and so on until the highest level $L$ is reached.  For $\ell = 1,\ldots,L$, let $\mu_{T_{\ell}}^{(\ell)}$ be the solution of \eqref{eq:SVGDDensity}, with the low-fidelity density $\pi^{(\ell)}$ replacing the target $\pi$, at time $T_{\ell}$ with initial density $\mu_{T_{\ell - 1}}^{(\ell - 1)}$ where the times $T_{\ell}$ are given by
    \begin{equation}
        T_{\ell} = \min \left\{ t \ge 0\ :\  \KL{\mu^{(\ell)}_{t}}{ \pi^{(\ell)}} \le \frac{\epsilon_{\ell}^2}{2} \right\} \, ,
    \label{eq:epsiloneq}
    \end{equation}
    where $\epsilon_1 \ge \epsilon_2 \ge \ldots \ge \epsilon_L$ and $\epsilon_L \le \epsilon$ is a sequence of tolerances.  Then, the continuous-time MLSVGD approximation is defined as
    \[
        \mu^{\mathrm{ML}} = \mu_{T_L}^{(L)} \, ,
    \]
    which gives the cost of MLSVGD as
    \[
        c_{\mathrm{ML}}(\epsilon) = \sum_{\ell=1}^L c_{\ell} T_{\ell} \, ,
    \]
    where both $L$ and $T_{\ell}$ will depend on $\epsilon$.  Since the KL divergence does not satisfy the triangle inequality, the following assumption for MLSVGD ensures that the KL divergence between levels converges as well, which is different from Assumption~\ref{asm:BiasRate}.

    \begin{assumption}
    There exists a constant $k_2 > 0$ independent of $\ell$ such that $\KL{\pi^{(\ell - 1)}}{\pi^{(\ell)}} \leq k_2 s^{-\alpha \ell}$\,, where $\alpha$ is the same rate as in Assumption~\ref{asm:BiasRate}.
    \label{asm:MLBiasRate}
    \end{assumption}

With these additional assumptions one can derive the cost complexity for MLSVGD \cite[Proposition 4]{MLSVGD} below.

    \begin{proposition}
    If Assumptions~\ref{asm:Costs}--\ref{asm:MLBiasRate} hold and $R_{\ell} \leq k_3 s^{-\alpha \ell}$ and
    \[
        \epsilon_{\ell} = \sqrt{2 k_1} s^{-\alpha \ell / 2}\, ,
    \]
    where
        \begin{equation}
        R_{\ell} = \int_{\R^d} \left(\mu^{(\ell-1)}_{T_{\ell-1}}(\btheta) - \pi^{(\ell-1)}(\btheta) \right) \log\left( \frac{\pi^{(\ell-1)}(\btheta)}{\pi^{(\ell)}(\btheta)}  \right)\ \mathrm{d}\btheta\,,
        \label{eq:ConstCL}
    \end{equation}
    then the costs of  
    MLSVGD to have $\dhell{\mu^{\mathrm{ML}}}{\pi} \le \epsilon$ can be bounded as
    \begin{equation}
        c_{\mathrm{ML}}(\epsilon) \le \frac{c_0 s^{2 \gamma}}{\lambda \gamma \log(s)} \log\left( s^{\alpha} + \frac{k_2 + k_3}{k_1} \right) \left( \frac{\sqrt{2k_1}}{\epsilon}  \right)^{2\gamma/\alpha}\, .
        \label{eq:MLSVGD:FastC}
    \end{equation}
    \label{prop:MLSVGD:MLSVGDCosts}
    \end{proposition}
    
    The cost complexity of MLSVGD scales at most as $\mathcal{O}(\epsilon^{-2\gamma/\alpha})$, whereas the cost complexity for single-level SVGD has an additional $\log \epsilon^{-1}$ factor.  Furthermore, the bound \eqref{eq:MLSVGD:FastC} is independent of the KL divergence of the initial density $\mu_0$ and instead only depends on the constant $k_2$ that measures the KL divergence between successive levels.

\section{Further analysis of MLSVGD}
\label{sec:3}

We now extend the analysis of MLSVGD to apply in settings where SVGD exhibits an exponential convergence rate with a varying parameter.

\subsection{Cost bound for MLSVGD}
\label{sec:3p1}

We now consider a relaxed assumption on the convergence rate that includes having $\lambda(t) \geq 0$ depend on time $t$ so that the multiplicative factor in \eqref{eq:KLExpBound} becomes $\mathrm e^{\lambda(t) t}$. The following assumption formalizes the time-dependent convergence factor as $r(t)$, which includes the case with factor $\mathrm e^{\lambda(t)t}$.

    \begin{assumption}
    There exists a decreasing function $r:[0,\infty) \to [0,1]$ such that $r(0) = 1$, $\lim_{t\to\infty}r(t) = 0$, and for an initial distribution $\nu_0$ 
    \[
        \KL{\nu_t}{\pi^{(\ell)}} \leq 
        r(t) \KL{\nu_0}{\pi^{(\ell)}}\,,\qquad \ell \in \mathbb{N}\,,
    \]
    holds, where $\nu_t$ is the solution of the mean-field SVGD equation~\eqref{eq:SVGDDensity} at time $t$.
    \label{asm:SVGDRateGeneral}
    \end{assumption}
    
    Note that by \eqref{eq:LyapunovKL}, for any fixed $\ell \in \mathbb{N}$ and initial distribution $\nu_0$, the relation in Assumption~\ref{asm:SVGDRateGeneral} holds, however, Assumption~\ref{asm:SVGDRateGeneral} is stronger in that it requires the inequality to hold uniformly for all levels $\ell$ and initial distributions $\nu_0$.  
    The uniformness implies that there exists a $\lambda > 0$ with $\lambda(t) \geq \lambda$ so that the analysis of \cite{MLSVGD} applies. However, we obtain a tighter bound in terms of constants in the following. 
    In the case where $r$ is not invertible due to a discontinuity, we define
    \begin{equation}
        r^{-1}(\epsilon) = \min \left\{  t \in [0,\infty)\ :\ r(t) \le \epsilon \right\} \, .
    \label{eq:InverseRate}
    \end{equation}

We now derive a result analogous to Proposition~\ref{prop:SVGDComplexity}.

\begin{proposition}
    If Assumptions~\ref{asm:Costs},\ref{asm:BiasRate},\ref{asm:SVGDRateGeneral} hold, then the costs of  
    SVGD to obtain $\mu^{\mathrm{SL}}$ with
    \[
        \dhell{\mu^{\mathrm{SL}}}{\pi} \leq \epsilon\, ,
    \]
    is bounded as
    \begin{equation}
        \mathrm{c}_{\mathrm{SL}}(\epsilon) \le c_0 s^{\gamma L} T_{\mathrm{SL}}  \le 2c_0 s^{\gamma} (2k_1)^{\gamma/\alpha}  r^{-1}\left( \frac{\epsilon^2}{2 \KL{\mu_0}{\pi^{(L)}}}  \right) \epsilon^{-2\gamma/\alpha}  \,.
    \label{eq:MLSVGD:SLSVGDCostGeneral}
    \end{equation}
\label{prop:SVGDcomplexityGeneral}
\end{proposition}

\begin{proof}
    By the triangle inequality for the Hellinger distance we have that
    \[
        \dhell{\mu^{\text{SL}}}{\pi} \le 
        \dhell{\mu^{\text{SL}}}{\pi^{(L)}}
        + \dhell{\pi^{(L)}}{\pi} ,
    \]
    so we will bound both of these terms independently by $\epsilon/2$.  By inequality~\eqref{eq:HellKLBound}, it is sufficient to bound the KL divergence because
    \begin{equation}
        \dhell{\mu^{\text{SL}}}{\pi^{(L)}} \le \sqrt{ \frac{\KL{\mu^{\text{SL}}}{\pi^{(L)}}}{2} },
    \end{equation}
    and similarly for $\dhell{\pi^{(L)}}{\pi}$.  By Assumption~\ref{asm:BiasRate} choose $L$ to be
    \begin{equation}
        L = \left\lceil \frac{1}{\alpha} \log_s \left( \frac{2 k_1}{\epsilon^2} \right) \right \rceil \le  \frac{1}{\alpha} \log_s \left( \frac{2 k_1}{\epsilon^2} \right) + 1,
    \label{eq:highfidlevel}
    \end{equation}
    so that 
    \begin{equation}
        \dhell{\pi^{(L)}}{\pi} \le \sqrt{ \frac{\KL{\pi^{(L)}}{\pi}}{2}} \le \sqrt{ \frac{k_1 s^{-\alpha L}}{2} } \le \frac{\epsilon}{2} .
    \end{equation}
    The time needed to integrate with SVGD to achieve $\dhell{\mu^{\text{SL}}}{\pi^{(L)}} \leq \epsilon/2$ is
    \[
        T_{\mathrm{SL}} = \min \left\{ t \ge 0\ :\ \dhell{\mu_t}{\pi^{(L)}} \le \frac{\epsilon}{2} \right\}\, .
    \]
    Again by inequality~\eqref{eq:HellKLBound},
    \[
        T_{\mathrm{SL}} \le \min \left\{ t\ge 0\ :\ \KL{\mu_t}{\pi^{(L)}} \le \frac{\epsilon^2}{2} \right\}\, .
    \]
    Now by Assumption~\ref{asm:SVGDRateGeneral}, the rate function $r$ is invertible, or by applying the definition~\eqref{eq:InverseRate} of $r^{-1}$, and the time needed to integrate with SVGD to achieve $\dhell{\mu^{\text{SL}}}{\pi^{(L)}} \leq \epsilon/2$ is bounded as
    \begin{equation}
        T_{\mathrm{SL}} \le r^{-1}\left( \frac{\epsilon^2}{2 \KL{\mu_0}{\pi^{(L)}}}  \right) .
    \end{equation}
    With Assumption~\ref{asm:Costs}, the total cost to integrate until time $T_{\mathrm{SL}}$ at level $L$ is therefore bounded as
    \[
        \mathrm{c}_{\mathrm{SL}}(\epsilon) \le c_0 s^{\gamma L} T_{\mathrm{SL}}  \le 2c_0 s^{\gamma} (2k_1)^{\gamma/\alpha}  r^{-1}\left( \frac{\epsilon^2}{2 \KL{\mu_0}{\pi^{(L)}}}  \right) \epsilon^{-2\gamma/\alpha} \, .
    \]
\end{proof}

As in Proposition~\ref{prop:SVGDComplexity} we see that the cost complexity depends on the tolerance $\epsilon$, the KL divergence of the initial distribution $\mu_0$ from the high-fidelity density $\pi^{(L)}$, as well as the SVGD convergence rate.  Because the rate function $r$ is decreasing, its inverse $r^{-1}$ is also decreasing and so a larger initial KL divergence will require a longer integration time.  We also derive a new cost complexity for the more general convergence behavior for  
MLSVGD in the following proposition.

\begin{proposition}
    If Assumptions~\ref{asm:Costs},~\ref{asm:BiasRate},~\ref{asm:MLBiasRate}, and~\ref{asm:SVGDRateGeneral} hold and $R_{\ell} \leq k_3 s^{-\alpha \ell}$, then by setting $\epsilon_{\ell} = \sqrt{2k_1}s^{-\alpha \ell/2}$ for $\ell=1,\ldots,L$, the costs of  
    MLSVGD to have $\dhell{\mu^{\mathrm{ML}}}{\pi} \le \epsilon$ can be bounded as
    \begin{equation}
        c_{\mathrm{ML}}(\epsilon) 
        \le
        \frac{c_0 s^{2\gamma} (2k_1)^{\gamma/\alpha}}{s^{\gamma} - 1}  r^{-1}\left( \frac{1}{s^{\alpha} + (k_2 + k_3)/k_1} \right) \epsilon^{-2\gamma/\alpha}\, .
        \label{eq:MLSVGD:FastCGeneral}
    \end{equation}
    \label{prop:MLSVGD:CostsGeneral}
\end{proposition}

\begin{proof}
    As in Equation~\eqref{eq:highfidlevel} in the proof of Proposition~\ref{prop:SVGDcomplexityGeneral} we select the level $L$ as
    \begin{equation}
        L = \left\lceil \frac{1}{\alpha} \log_s \left( \frac{2 k_1}{\epsilon^2} \right) \right\rceil \le \frac{1}{\alpha} \log_s \left( \frac{2 k_1}{\epsilon^2} \right) + 1,
    \end{equation}
    so that $\dhell{\pi^{(L)}}{\pi} \le \epsilon/2$.
    Note that by setting $\epsilon_{\ell} = \sqrt{2k_1} s^{-\alpha \ell / 2}$ for $\ell=1,\ldots,L$ in~\eqref{eq:epsiloneq} we have that
    \[
        \frac{\epsilon_L^2}{2} = k_1 s^{-\alpha L} \le \frac{\epsilon}{2} \, ,
    \]
    by the choice of the high-fidelity level $L$~\eqref{eq:highfidlevel}.

    By Assumption~\ref{asm:Costs}, the total cost for MLSVGD is bounded by
    \begin{equation}
        c_{\mathrm{ML}}(\epsilon) \le \sum_{\ell=1}^L c_0 s^{\gamma \ell} T_{\ell}\, ,
    \end{equation}
    where it remains to bound the integration times $T_{\ell}$ at each level. By Assumption~\ref{asm:SVGDRateGeneral} and Equation~\eqref{eq:recursive}, we have
    \begin{equation}
    \begin{split}
        \KL{\mu^{(\ell)}_{T_{\ell}}}{ \pi^{(\ell)}} 
        &\le 
        r(T_{\ell}) \KL{\mu_{T_{\ell-1}}^{(\ell-1)}}{\pi^{(\ell)}} \\
        &=
        r( T_{\ell} )
        \left( \KL{\mu^{(\ell-1)}_{T_{\ell-1}}}{ \pi^{(\ell-1)}} 
        +  \KL{\pi^{(\ell - 1)}}{ \pi^{(\ell)} } + R_{\ell}  \right)\, ,
    \end{split}
    \label{eq:recursive}
    \end{equation}
    giving a recursive bound on the KL divergence in terms of the KL divergence at the previous level.  By the definition~\eqref{eq:epsiloneq} of the integration times $T_{\ell}$ at level $\ell$, we know that
    \begin{equation}
        \KL{\mu^{(\ell)}_{T_{\ell}}}{ \pi^{(\ell)}} 
        \le
        \frac{\epsilon_{\ell}^2}{2}\, , 
    \label{eq:KLBoundEpsilon}
    \end{equation}
    is satisfied for each level $\ell=1,\ldots,L$.  Using~\eqref{eq:KLBoundEpsilon} at level $\ell-1$ gives
    \begin{equation}
        \KL{\mu^{(\ell)}_{T_{\ell}}}{ \pi^{(\ell)}} 
        \le
        r(T_{\ell}) \left( \frac{\epsilon_{\ell-1}^2}{2} + \KL{\pi^{(\ell - 1)}}{ \pi^{(\ell)} }  + R_{\ell} \right) \, .
    \label{eq:KLBound1}
    \end{equation}
    Note that by~\eqref{eq:KLBoundEpsilon} we know that the left-hand-side of~\eqref{eq:KLBound1} is guaranteed to be bounded above by $\epsilon_{\ell}^2/2$, but the same is not necessarily true for the right-hand-side which is an upper bound.  Instead define $T_{\ell}^{\prime}$ as
    \begin{equation}
        T_{\ell}^{\prime} 
        =
        \min \left\{t \ge 0\ :\ r(t) \left( \frac{\epsilon_{\ell-1}^2}{2} + \KL{\pi^{(\ell - 1)}}{ \pi^{(\ell)} } + R_{\ell} \right)  \le \frac{\epsilon_{\ell}^2}{2} \right\}\, ,
    \label{eq:TBoundAlt}
    \end{equation}
    for each level $\ell=1,\ldots,L$, which is finite by the assumption that $r(t) \to 0$ (Assumption~\ref{asm:SVGDRateGeneral}).  By~\eqref{eq:KLBound1} and because $r$ is monotonically decreasing we know that $T_{\ell} \le T_{\ell}^{\prime}$.  Solving directly gives
    \begin{equation}
        T_{\ell}^{\prime} 
        \le
        r^{-1}\left( \frac{\epsilon_{\ell}^2}{\epsilon_{\ell-1}^2 + 2 \KL{\pi^{(\ell - 1)}}{ \pi^{(\ell)} }  + 2 R_{\ell} } \right) \, .
    \label{eq:TBoundAlt1}
    \end{equation}
    We now use the fact that $r^{-1}$ is decreasing as well as Assumption~\ref{asm:MLBiasRate} and the assumption that $R_{\ell} \le k_3 s^{-\alpha \ell}$ to bound
    \begin{equation*}
        r^{-1}\left( \frac{\epsilon_{\ell}^2}{\epsilon_{\ell-1}^2 + 2 \KL{\pi^{(\ell - 1)}}{ \pi^{(\ell)} }  + 2 R_{\ell} } \right) \le r^{-1}\left( \frac{\epsilon_{\ell}^2}{\epsilon_{\ell-1}^2 + 2k_2 s^{-\alpha \ell} + 2 k_3 s^{-\alpha \ell}  } \right)\, .
    \end{equation*}
    Therefore, by substituting $\epsilon_{\ell} = \sqrt{2k_1} s^{-\alpha \ell/2}$ (and similarly for $\epsilon_{\ell-1}$) we can bound $T_{\ell}^{\prime}$, and hence $T_{\ell}$, with
    \begin{equation*}
        T_{\ell} \le r^{-1}\left( \frac{ 2k_1s^{-\alpha \ell}  }{ 2k_1 s^{\alpha} s^{-\alpha \ell} + 2k_2 s^{-\alpha \ell} + 2 k_3 s^{-\alpha \ell}  } \right)\, .
    \end{equation*}
    Simplifying gives the bound
    \begin{equation}
        T_{\ell} \le r^{-1}\left( \frac{1}{s^{\alpha} + (k_2 + k_3)/k_1} \right) \, ,
    \label{eq:TBoundAlt3}
    \end{equation}
    which is independent of the tolerance $\epsilon$.  The total cost can now be bounded by
    \begin{equation}
        c_{\mathrm{ML}}(\epsilon) 
        \le
        \sum_{\ell=1}^L c_0 s^{\gamma \ell}  r^{-1}\left( \frac{1}{s^{\alpha} + (k_2 + k_3)/k_1} \right) \, ,
    \label{eq:CostBoundAltNew}
    \end{equation}
    which we may again compute explicitly
    \begin{equation}
    \begin{split}
        c_{\mathrm{ML}}(\epsilon) 
        &\le
        c_0 s^{\gamma} r^{-1}\left( \frac{1}{s^{\alpha} + (k_2 + k_3)/k_1} \right) \frac{s^{\gamma L} - 1}{s^{\gamma} - 1}\\
        &\le
        c_0 s^{\gamma} r^{-1}\left( \frac{1}{s^{\alpha} + (k_2 + k_3)/k_1} \right) \frac{s^{\gamma L}}{s^{\gamma} - 1}
        \, ,
    \end{split}
    \end{equation}
    and we have again added 1 in the numerator of the last term for convenience.  Substituting the upper bound~\eqref{eq:highfidlevel} on the level $L$ and simplifying terms gives the final upper bound on the improved cost complexity of the  
    MLSVGD approximation $\mu^{\mathrm{ML}}$
    \begin{equation}
        c_{\mathrm{ML}}(\epsilon) 
        \le
        \frac{c_0 s^{2\gamma} (2k_1)^{\gamma/\alpha}}{s^{\gamma} - 1}  r^{-1}\left( \frac{1}{s^{\alpha} + (k_2 + k_3)/k_1} \right) \epsilon^{-2\gamma/\alpha} \, .
    \label{eq:MLSVGDCostUpperBoundNew}
    \end{equation}
\end{proof}

By setting $r(t) = \mathrm{e}^{-\lambda t}$ as in Assumption~\ref{asm:SVGDRate}, one can recover the cost complexities stated in Section~\ref{sec:2p3}.   
When compared to \eqref{eq:MLSVGD:SLSVGDCostGeneral}, if SVGD is slow to converge then the  
MLSVGD can spend most of the integration time at the lower levels, which can be faster to integrate, in order to find a good initial density for integrating with respect to the highest level $L$ and so potentially achieve speedups.  In contrast, if SVGD converges quickly then the low-fidelity densities will be less beneficial and both costs will be comparable.

\subsection{Cost complexity for Bayesian inverse problems}
\label{sec:3p3}
The results from Section~\ref{sec:3p1}  
are applicable in Bayesian inverse problem settings. Recall that typically in Bayesian inverse problems, the sequence of posterior distributions $(\pi^{(\ell)})$ is obtained via a sequence of approximate parameter-to-observable maps $(G^{(\ell)})_{\ell=1}^{\infty}$ with $G^{(\ell)}(\btheta) \to G(\btheta)$ pointwise for every $\btheta \in \Theta$, so that the sequence of densities $(\pi^{(\ell)})$ converges pointwise as well. As shown in \cite{MLSVGD}, the following assumption on the parameter-to-observable maps ensures that the KL divergences of the densities converges as required in Assumptions~\ref{asm:BiasRate} and ~\ref{asm:MLBiasRate}.

    \begin{assumption}
    The error of the approximate parameter-to-observable $G^{(\ell)}$ map at level $\ell \ge 1$ is bounded by
    \begin{equation}
        \norm{G(\btheta) - G^{(\ell)}(\btheta)}_{L^2(\pi_0)} \le b_0s^{-\alpha \ell} \, ,
    \label{eq:modelerror}
    \end{equation}
    where $\alpha,b_0 > 0$ and $s > 1$ are constants with $s$ the same as in Assumption~\ref{asm:Costs} and $\norm{\cdot}_{L^2(\pi_0)}$ is the $L^2$ norm over $\pi_0$.
    \label{asm:modelerror}
    \end{assumption}
    
As long as the SVGD approximations $\mu_{T_{\ell}}^{(\ell)}$ remain absolutely continuous at each level with respect to the prior density $\pi_0$, the remainders $R_{\ell}$ defined in \eqref{eq:ConstCL} can be bounded and thus the cost bound with the same rate as in Proposition~\ref{prop:MLSVGD:CostsGeneral} applies for approximating the Bayesian posterior; see \cite[Theorem~1]{MLSVGD}. We now state this result formally.

\begin{proposition}
Let Assumptions~\ref{asm:Costs},~\ref{asm:SVGDRateGeneral},and~\ref{asm:modelerror} hold. Furthermore, assume that there exists a constant $b_3 > 0$ independent of $\ell$ such that
    \begin{equation} \label{eq:CostBIP:AsmCover}\mu^{(\ell)}_{T_{\ell}}(\btheta) \le b_3 \pi_0(\btheta)\, ,\quad \btheta \in \Theta \, ,
    \end{equation}
    for all $\ell \ge 1$.
    Then, the cost complexity of finding $\mu^{\mathrm{ML}}$ with $\dhell{\mu^{\mathrm{ML}}}{ \pi} \le \epsilon$ is 
    \begin{equation}\label{eq:BIPCost:Bound}
        c_{\mathrm{ML}}(\epsilon) 
        \le
        \frac{c_0 s^{2 \gamma} (3b_1b_2b_0)^{\gamma/\alpha}}{s^{\gamma} - 1}
        r^{-1}\left( \frac{1}{s^{\alpha} + (1 + s^{\alpha})(4 + 3b_3/b_2) } \right) 
        \epsilon^{-2\gamma/\alpha}\, ,
    \end{equation}
    where the constants $b_1,b_2$ are independent of $\epsilon$. 
    \label{thm:BayesCostComp}
\end{proposition}

\begin{proof}
    Because Assumption~\ref{asm:modelerror} holds, by \cite[Lemmas~7 and 8]{MLSVGD} we know that
    Assumptions~\ref{asm:BiasRate} and~\ref{asm:MLBiasRate} hold with $k_1 = \frac{3}{2}b_0b_1b_2$ and $k_2 = \frac{3}{2}b_0b_1b_2(1 + s^{\alpha})$. 
    Here the constants $b_1,b_2$ are such that
    \begin{equation}
        b_1 \ge \norm{\Gamma^{-1}\left( 2\by - G^{(\ell_1)} - G^{(\ell_2)}\right)}_{L^2(\pi_0)} \, ,
    \label{eq:b1}
    \end{equation}
    for all $\ell_1,\ell_2 \ge 1$ and
    \begin{equation}
        b_2 \ge \sup_{\ell \ge 1} \frac{1}{Z_{\ell}} \, .
    \label{eq:b2}
    \end{equation}
    Thus, we just need to verify that $R_{\ell} \le k_3 s^{-\alpha \ell}$ holds for a constant $k_3$ to apply Proposition~\ref{prop:MLSVGD:CostsGeneral}. 
    Using definitions given in \eqref{eq:Prelim:TargetPDF} and \eqref{eq:ConstCL}, we make the following transformations
    \begin{equation}\label{eq:BIP:ProofOne}
    \begin{split}
        R_{\ell} 
        &= 
        \int_{\Theta} \left( \mu^{(\ell-1)}_{T_{\ell-1}}(\btheta) - \pi^{(\ell - 1)}(\btheta) \right) \log \left( \frac{\pi^{(\ell-1)}(\btheta)}{\pi^{(\ell)}(\btheta)} \right)\ \mathrm{d}\btheta       \\
        &=  
        \int_{\Theta} \left( \mu^{(\ell-1)}_{T_{\ell-1}}(\btheta) - \pi^{(\ell - 1)}(\btheta) \right) \log \left(  \frac{Z_{\ell} \exp\left( -\frac{1}{2}\|\by - G^{(\ell-1)}(\btheta)\|_{\Gamma^{-1}}^2 \right)}{Z_{\ell - 1} \exp\left( -\frac{1}{2}\|\by - G^{(\ell)}(\btheta)\|_{\Gamma^{-1}}^2 \right) }  \right) \ \mathrm{d}\btheta \\
        &=  
        \int_{\Theta} \left( \mu^{(\ell-1)}_{T_{\ell-1}}(\btheta) - \pi^{(\ell - 1)}(\btheta) \right) \log \left(  \frac{ \exp\left( -\frac{1}{2} \|\by - G^{(\ell-1)}(\btheta) \|_{\Gamma^{-1}}^2 \right)}{ \exp\left( -\frac{1}{2}\|\by - G^{(\ell)}(\btheta)\|_{\Gamma^{-1}}^2 \right) }  \right) \ \mathrm{d}\btheta  ,
    \end{split}
    \end{equation}
    where $Z_{\ell}$ and $Z_{\ell-1}$ are the normalizing constants of $\pi^{(\ell)}$ and $\pi^{(\ell-1)}$, respectively, so that
    \begin{equation}
        Z_{\ell} = \int_{\Theta} \exp\left( -\frac{1}{2}\|\by - G^{(\ell)}(\btheta)\|_{\Gamma^{-1}}^2 \right) \pi_0(\btheta)\ \mathrm{d}\btheta \, .
    \label{eq:zell}
    \end{equation}
    The last line of \eqref{eq:BIP:ProofOne} follows from the fact that
    \begin{equation}
         \int_{\Theta} \left( \mu^{(\ell-1)}_{T_{\ell-1}}(\btheta) - \pi^{(\ell - 1)}(\btheta) \right) \log \left(  \frac{Z_{\ell} }{Z_{\ell - 1}  }  \right) \ \mathrm{d}\btheta = 0
    \end{equation}
    since $\frac{Z_{\ell} }{Z_{\ell - 1}  }$ is constant in $\btheta$ and $\pi^{(\ell - 1)}$ and $\mu^{(\ell-1)}_{T_{\ell-1}}$ both integrate to one.  Simplifying the expression for $R_{\ell}$ gives
    \begin{equation}
        R_{\ell} = \frac{1}{2}\int_{\Theta} \left( \mu^{(\ell-1)}_{T_{\ell-1}}(\btheta) - \pi^{(\ell - 1)}(\btheta) \right)  \left( \|\by - G^{(\ell)}(\btheta)\|_{\Gamma^{-1}}^2 - \|\by - G^{(\ell-1)}(\btheta) \|_{\Gamma^{-1}}^2 \right) \ \mathrm{d}\btheta \, .
    \end{equation}
    By taking the absolute value and applying the triangle inequality we have that
    \begin{equation}
    \begin{split}
        R_{\ell} &\le \frac{1}{2} \int_{\Theta} \left| \| \by - G^{(\ell)}(\btheta)\|_{\Gamma^{-1}}^2 - \| \by - G^{(\ell-1)}(\btheta)\|_{\Gamma^{-1}}^2  \right|  \mu^{(\ell-1)}_{T_{\ell-1}}(\btheta)  \ \mathrm{d}\btheta  \\
        &\quad + \frac{1}{2} \int_{\Theta} \left| \| \by - G^{(\ell)}(\btheta)\|_{\Gamma^{-1}}^2 - \| \by - G^{(\ell-1)}(\btheta)\|_{\Gamma^{-1}}^2  \right|  \pi^{(\ell-1)}(\btheta)  \ \mathrm{d}\btheta\, .
    \end{split}    
    \end{equation}
    Additionally, since
    \[
        \exp\left( -\frac{1}{2}\|\by - G^{(\ell)}(\btheta)\|_{\Gamma^{-1}}^2 \right) \le 1 \, ,
    \]
    we have
    \begin{equation}
        \pi^{(\ell)}(\btheta) = \frac{1}{Z_{\ell}}\exp\left( -\frac{1}{2}\|\by - G^{(\ell)}(\btheta)\|_{\Gamma^{-1}}^2 \right) \pi_0(\btheta) \le \frac{1}{Z_{\ell}}\pi_0(\btheta) \, .
    \label{eq:piboundzell}
    \end{equation}
    Therefore, by combining~\eqref{eq:piboundzell} for $\pi^{(\ell-1)}(\btheta) \le \pi_0(\btheta)/Z_{\ell-1}$
    with \eqref{eq:CostBIP:AsmCover} we get 
        \begin{equation}
    \begin{split}
        R_{\ell} &\le \frac{1}{2} \int_{\Theta} \left| \| \by - G^{(\ell)}(\btheta)\|_{\Gamma^{-1}}^2 - \| \by - G^{(\ell-1)}(\btheta)\|_{\Gamma^{-1}}^2  \right|  \mu^{(\ell-1)}_{T_{\ell-1}}(\btheta)  \ \mathrm{d}\btheta  \\
        &\quad + \frac{1}{2} \int_{\Theta} \left| \| \by - G^{(\ell)}(\btheta)\|_{\Gamma^{-1}}^2 - \| \by - G^{(\ell-1)}(\btheta)\|_{\Gamma^{-1}}^2  \right|  \pi^{(\ell-1)}(\btheta)  \ \mathrm{d}\btheta \\
        &\le \frac{b_3}{2} \int_{\Theta} \left| \| \by - G^{(\ell)}(\btheta)\|_{\Gamma^{-1}}^2 - \| \by - G^{(\ell-1)}(\btheta)\|_{\Gamma^{-1}}^2  \right|  \pi_{0}(\btheta)  \ \mathrm{d}\btheta  \\
        &\quad + \frac{1}{2 Z_{\ell - 1}} \int_{\Theta} 
        \left| \| \by - G^{(\ell)}(\btheta)\|_{\Gamma^{-1}}^2 - \| \by - G^{(\ell-1)}(\btheta)\|_{\Gamma^{-1}}^2  \right|  \pi_{0}(\btheta)  \ \mathrm{d}\btheta \, .
    \end{split}
    \end{equation}
    Re-writing the expression inside the absolute value in the integrand gives
    \begin{equation}
    \begin{split}
    & \| \by - G^{(\ell)}(\btheta)\|_{\Gamma^{-1}}^2 - \| \by - G^{(\ell-1)}(\btheta)\|_{\Gamma^{-1}}^2   \\
        &=
    \norm{G^{(\ell)}(\btheta) - G^{(\ell-1)}(\btheta) + G^{(\ell-1)}(\btheta)-\by}_{\Gamma^{-1}}^2 - \| \by - G^{(\ell-1)}(\btheta)\|_{\Gamma^{-1}}^2\\
    &=
    \inner{G^{(\ell)}(\btheta) - G^{(\ell-1)}(\btheta)}{\Gamma^{-1}(G^{(\ell)}(\btheta) - G^{(\ell-1)}(\btheta))} \\&\qquad\qquad\qquad
    + 2 \inner{G^{(\ell)}(\btheta) - G^{(\ell-1)}(\btheta)}{\Gamma^{-1}(G^{(\ell-1)}(\btheta) - \by)}\\
    &= 
    \inner{G^{(\ell)}(\btheta) - G^{(\ell-1)}(\btheta)}{\Gamma^{-1}(G^{(\ell)}(\btheta) + G^{(\ell-1)}(\btheta) - 2\by)}\\
    &\le 
    \norm{G^{(\ell)}(\btheta) - G^{(\ell-1)}(\btheta)} \cdot \norm{\Gamma^{-1}(2\by - G^{(\ell)}(\btheta) - G^{(\ell-1)(\btheta)})} \, , 
    \end{split}
    \end{equation}
    where we have applied the Cauchy-Schwarz inequality to obtain the last line.
    From the Cauchy-Schwarz inequality on $L^2(\pi_0)$
    \begin{equation}
    \begin{split}
    &\int_{\Theta} 
        \left| \| \by - G^{(\ell)}(\btheta)\|_{\Gamma^{-1}}^2 - \| \by - G^{(\ell-1)}(\btheta)\|_{\Gamma^{-1}}^2  \right|  \pi_{0}(\btheta)  \ \mathrm{d}\btheta \\
    &\le
        \int_{\Theta} \norm{G^{(\ell)}(\btheta) - G^{(\ell-1)}(\btheta)} \cdot \norm{\Gamma^{-1}(2\by - G^{(\ell)}(\btheta) - G^{(\ell-1)(\btheta)})} \pi_0(\btheta)\ \mathrm{d}\btheta \\
        &\le
        \norm{G^{(\ell)} - G^{(\ell-1)}}_{L^2(\pi_0)} \cdot \norm{\Gamma^{-1}(2\by - G^{(\ell)} - G^{(\ell-1)})}_{L^2(\pi_0)} \\
        &\le
        b_1 \norm{G^{(\ell)} - G^{(\ell-1)}}_{L^2(\pi_0)} \\
        &\le
        b_1\norm{G - G^{(\ell)}}_{L^2(\pi_0)} + b_1\norm{G - G^{(\ell-1)}}_{L^2(\pi_0)}\, .
    \end{split}
    \label{eq:CSproof}
    \end{equation}
    From Assumption~\ref{asm:modelerror} we have
    \[
        b_1\norm{G - G^{(\ell)}}_{L^2(\pi_0)} + b_1\norm{G - G^{(\ell-1)}}_{L^2(\pi_0)} \le b_0b_1s^{-\alpha \ell} (1 + s^{\alpha}) \, ,
    \]
    and therefore
    \begin{equation}
        R_{\ell} \le \frac{1}{2}b_0b_1 (1 + s^{\alpha})\left(b_3 + b_2\right)s^{-\alpha \ell}  ,
    \end{equation}
    so that $R_{\ell} \le k_3 s^{-\alpha \ell}$ with $k_3 = \frac{b_0 b_1}{2}\left(b_2 + b_3\right)(1 + s^{\alpha})$. Thus, Proposition~\ref{prop:MLSVGD:CostsGeneral} applies and gives the bound \eqref{eq:BIPCost:Bound}. 
    \end{proof}

\section{Numerical example: Ice sheet modeling of the Arolla glacier}\label{sec:NumExp}
To demonstrate the applicability and performance of MLSVGD, we formulate and solve an inverse problem governed by a Stokes ice sheet model. In particular, we infer the basal sliding coefficient field 
from pointwise surface velocity observations. The problem formulation and adjoint-based derivatives computation follows the work in~\cite{Noemi2014}. The numerical computations are carried out in Python using FEniCS~\cite{fenics} and hIPPYlib \cite{VillaPetraGhattas16,VillaPetraGhattas18,VillaPetraGhattas21,https://doi.org/10.48550/arxiv.2112.00713}.
All reported runtimes were measured on Intel Xeon Platinum 8268 24C 205W 2.9GHz Processor.  The computation of the gradients $\nabla \log \pi^{(\ell)}(\btheta^{[j]}_t)$ was parallelized over 32 cores.

\subsection{Nonlinear Stokes forward model}

For the numerical studies, we use an ice sheet model problem that uses the Arolla (Haut Glacier d’Arolla) geometry and setup from the ISMIP-HOM benchmark collection~\cite{ISMIP-HOM}. That is, the glacier is considered a sliding mass of ice whose velocity is determined primarily by the force of gravity and the friction against the underlying rock.  The ice flow is modeled as a non-Newtonian, viscous, incompressible fluid.  The velocity field $\bfu$ over the domain $\Omega \subset \mathbb{R}^2$, as shown in Figure~\ref{fig:arolla_domain}, is governed by the following Stokes equations
\begin{equation}
\begin{split}
\nabla \cdot {\boldsymbol u} = 0, &\quad \text{ in } \Omega\, ,\\
-\nabla \cdot {\boldsymbol \sigma_u} = \rho {\boldsymbol g}, & \quad \text{ in } \Omega \, .
\end{split}
\label{eq:arolla_pde}
\end{equation}
The boundary conditions along the top and bottom of the glacier are given as 
\begin{equation}
\begin{split}
\boldsymbol{n}^{\top}\left({\boldsymbol \sigma_u} {\boldsymbol n}  + \omega {\boldsymbol u}\right)=0, &\quad \text{ on } \Gamma_b \, ,\\
{\bf T} {\boldsymbol \sigma_u} {\boldsymbol n}  + \exp(\beta) {\bf T}{\boldsymbol u}  = {\boldsymbol 0}, &\quad \text{ on } \Gamma_b \, , \\
{\boldsymbol \sigma_u} {\boldsymbol n} = {\boldsymbol 0}, &\quad \text{ on } \Gamma_t \, .
\end{split}
\label{eq:arolla_bc}
\end{equation}
The density of the ice is $\rho= 910\  [\mathrm{kg}/\mathrm{m}^3]$ and the downwards gravitational force is ${\boldsymbol g} = (0, -9.81)\ [\mathrm{m}/\mathrm{s}^2]$.  For the boundary conditions, $\Gamma_b$ represents the bottom part of the domain where the ice slides across the bedrock and $\Gamma_t$ represents the top part of the domain; see Figure~\ref{fig:arolla_domain}.  The vector ${\boldsymbol n}$ represents the outward unit normal vector and ${\bf T} = \mathbf{I} - {\boldsymbol n}  {\boldsymbol n}^{\top}$ is the tangential projection.  In the first boundary condition where $\boldsymbol{n}^{\top}\left({\boldsymbol \sigma_u} {\boldsymbol n}  + \omega {\boldsymbol u}\right)=0$ on $\Gamma_b$ we set the parameter $\omega=10^{6}$ and is meant to approximate the no out-flow condition ${\boldsymbol u} \cdot {\boldsymbol n} = 0$, which is difficult to enforce directly due to the curvature of the domain $\Omega$.  The stress tensor is
    \[
        \bsigma_{\bu} = \btau_{\bu} - \eye p,
    \]
    with pressure $p$ and deviatoric stress tensor
    \[
        \btau_{\bu} = 2\eta(\bu) \dot{\beps}(\bu) \, ,
    \]
    with effective viscosity
    \[
        \eta(\bu) = \frac{1}{2}A^{-\frac{1}{n}} \dot{\beps}_{\Romanbar{II}}^{\frac{1 - n}{2n}} \, .
    \]
    The constants are Glen's flow law exponent $n = 3$ and the flow rate factor $A = 10^{-16}\  [\mathrm{Pa}^{-n}\mathrm{a}^{-1}]$ (Pascals and years, respectively).  The strain rate tensor is
    \[
        \dot{\beps} = \frac{1}{2}\left( \nabla \bu + \nabla \bu^{\top}\right) \, ,
    \]
    as well as the second invariant 
    \[
        \dot{\beps}_{\Romanbar{II}} = \frac{1}{2}\mathrm{tr}\left( \dot{\beps}_{\bu}^2 \right) \,,
    \]
where $\operatorname{tr}$ denotes the trace operator. 
The parameter of interest is the log basal sliding coefficient field $\beta:[0, 5000] \to \mathbb{R}$, which models the friction of the ice sheet across the underlying bedrock and relates tangential traction to the tangential velocity.  

To solve~\eqref{eq:arolla_pde}, we discretize~\eqref{eq:arolla_pde}--\eqref{eq:arolla_bc} using Taylor-Hood finite elements on a triangular mesh where the velocity is discretized with quadratic Lagrange elements and the pressure is discretized with linear Lagrange elements.  We consider one high-fidelity model and two low-fidelity models by coarsening the mesh.  The high-fidelity forward model $F^{(3)}$ ($L=3$) maps the log basal sliding coefficient field $\beta$ to the velocity field solution $\bu$ using 3,602 and 501 degrees of freedom for the velocity and pressure components, respectively.  The coarsest low-fidelity model $F^{(1)}$ uses 448 and 73 degrees of freedom for the velocity and pressure and the second low-fidelity model $F^{(2)}$ uses 1002 and 151 degrees of freedom, respectively.  To solve the discretized PDE we use a constrained Newton solver with the gradient tolerance set to $10^{-6}$.

\begin{figure}
    \centering
    \input{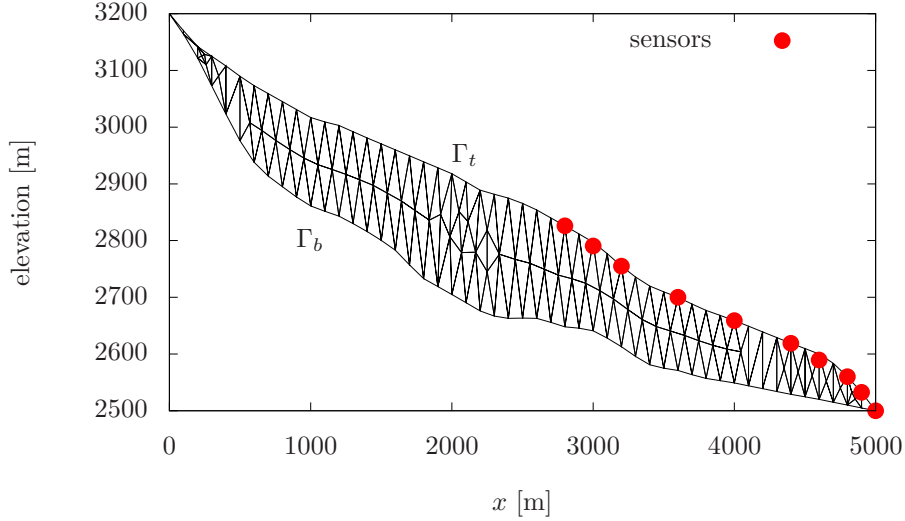}
    \caption{The domain $\Omega$ of Haut Glacier d’Arolla from the ISMIP-HOM benchmark collection~\cite{ISMIP-HOM}. The red dots represent the location of the measurements.}
    \label{fig:arolla_domain}
\end{figure}

\subsection{Problem setup}

We are interested in inferring a discretized log basal sliding coefficient field $\beta$, which effectively determines the velocity of the ice as it slides along the bedrock. We discretize the coefficient field $\beta:[0,5000] \to \mathbb{R}$ with a vector ${\boldsymbol \beta} \in \mathbb{R}^d$ ($d = 25$) that we aim to infer from data of the parameter-to-observable map.   The parameter vector ${\boldsymbol \beta} \in \mathbb{R}^{25}$ corresponds to $25$ equally-spaced pointwise evaluations of the coefficient field $\beta$ throughout the domain $[0,5000]$.  In particular, let $\mathcal{I}^{\mathrm{int}}$ denote the interpolation operator that maps a vector ${\boldsymbol \beta} \in \mathbb{R}^{25}$ to its piecewise linear interpolant $\bar{\beta}:[0,5000] \to \mathbb{R}$ defined at the nodes $x_i = 5000(i-1)/24$ by $\bar{\beta}(x_i) = {\boldsymbol \beta}_i$ for $i=1,\ldots,25$.  Given the piecewise linear interpolant $\bar{\beta}$, the forward models $F^{(\ell)}$ for $\ell=1,2,3$ map the parameter to the corresponding velocity field $\bu$.  Finally, the observation operator $\mathcal{B}^{\mathrm{obs}}$ maps the solution $\bu$ of~\eqref{eq:arolla_pde}, given by the output of the forward models $F^{(\ell)}$, to a 20 dimensional vector of horizontal and vertical velocity measurements at 10 sensor locations throughout the right side of the domain along the top of the glacier as shown in Figure~\ref{fig:arolla_domain}.  The full parameter-to-observable map $G^{(\ell)}:\mathbb{R}^{25} \to \mathbb{R}^{20}$ is
\[
    G^{(\ell)} = \mathcal{B}^{\mathrm{obs}} \circ F^{(\ell)} \circ \mathcal{I}^{\mathrm{int}}, \; \ell=1,2,3.
\]

Now consider the true parameter vector ${\boldsymbol \beta}^* = [\beta_1^*, \dots, \beta_{25}^*]^{\top} \in \mathbb{R}^{25}$ as given by taking pointwise evaluations 
\begin{equation}\label{eq:TrueBetaComponents}
\beta^*_i = \beta_{\mathrm{true}}(x_i)\,,\qquad x_i = 5000 (i - 1)/24\,,\quad i = 1,\ldots,25\,,
\end{equation}
where
\[
    \beta_{\mathrm{true}}(x) = \log \begin{cases}
    1000 + 1000\sin\left(\frac{3\pi x}{5000} \right) + \zeta & \text{ if } 0 \le x < 2500,\\
    1000\left(16 - \frac{x}{250} \right) + \zeta & \text{ if } 2500 \le x < 4000,\\
    1000 + \zeta & \text{ if } 4000 \le x < 5000\, , 
    \end{cases}
\]
and $\zeta = 10^{-6}$ is a small positive constant to ensure that the log basal coefficient field remains bounded.
We generate synthetic observations $\by\in \mathbb{R}^{20}$ with
\[
    \by = G^{(L+1)}( {\boldsymbol \beta}^* ) + {\boldsymbol \eta},\quad {\boldsymbol \eta} \sim N(\zeros, \Gamma)\, ,
\]
where the noise covariance matrix $\Gamma$ is diagonal with $\sigma_{\mathrm{vertical}} = 3$ and $\sigma_{\mathrm{horizontal}} = 18$ corresponding to the vertical and horizontal velocity measurements, respectively.  Here the level $L+1$ (a further refinement of the high-fidelity mesh) is used to compute the observed data $\by$.  

The prior $\pi_0$ is Gaussian with mean perturbed from the true parameters ${\boldsymbol \beta}^*$ and diagonal covariance matrix with variance $0.05$ along each diagonal.  The starting distribution for SVGD and MLSVGD is the 25-dimensional standard normal distribution.  The gradients of the log posterior density are computed using adjoints with hIPPYlib~\cite{VillaPetraGhattas16,VillaPetraGhattas18,VillaPetraGhattas21}.  Our quantity of interest is the mean of the posterior distribution $\mathbb{E}_{\pi^{(L)}}[{\boldsymbol \beta} ]$ and we compute a reference value $\hat{\boldsymbol \beta}^{\mathrm{Ref}}$ by using the preconditioned Crank-Nicolson (pCN) method \cite{pCN}.  We run 100 independent chains and use a burn-in period of 10,000 samples for each chain to obtain $10^7$ total samples.  The parameter in the pCN algorithm is set to $10^{-2}$.

\subsection{SVGD algorithm with approximate gradients}
Consider an empirical measure 
\begin{equation}
        \hat{\mu}_{\tau}^{(N)} = \frac{1}{N} \sum_{i=1}^{N} \delta_{ \btheta^{[i]}_{\tau} }  \, ,
        \label{eq:SVGD_approx}
    \end{equation}
given by an ensemble of particles $\{\btheta_{\tau}^{[i]}\}_{i=1}^N$ with $\{\btheta_0^{[i]}\}_{i=1}^N \sim \mu_0$ and where $\delta_{\bfx}$ represents the Dirac-mass at $\bfx$. 
    Practical SVGD implementations alternate between using the ensemble of particles $\{\btheta_{\tau}^{[i]}\}_{i=1}^N$ at time $\tau$ to estimate the gradient \eqref{eq:SteinIdentity} 
    and using the estimated gradient to update the ensemble to obtain $\hat{\mu}_{\tau+1}$.  The gradient is estimated with Monte Carlo from the current ensemble of particles as
    \begin{equation}
        \hat{\bg}_{\tau}(\btheta) = -\frac{1}{N} \sum_{i=1}^N K(\btheta_{\tau}^{[i]}, \btheta) \nabla \log \pi^{(L)}(\btheta_{\tau}^{[i]}) + \nabla_1 K( \btheta_{\tau}^{[i]}, \btheta)\,.
    \label{eq:SVGDGradient}
    \end{equation}
    The SVGD algorithm then reuses the ensemble of particles and updates them according to the approximate gradient with step size $\delta$ as
    \begin{equation}
        \btheta_{\tau+1}^{[j]} =  \btheta_{\tau}^{[j]} - \delta \hat{\bg}_{\tau}(\btheta_{\tau}^{[j]}) \, , \quad j = 1,\ldots,N\,.
        \label{eq:SVGDUpdate}
    \end{equation}
    Because the Hellinger distance from the high-fidelity density $\pi^{(L)}$ at iteration $\tau$ is unknown, the integration time given by \eqref{eq:SLTime} cannot be determined practically.  Instead the stopping criteria is that the average norm of the gradient $\bar{g}_{\tau}$, defined as
    \[
        \bar{g}_{\tau} = \frac{1}{N}\sum_{j=1}^{N} \norm{\hat{\bg}_{\tau}(\btheta_{\tau}^{[j]}) } \, ,
    \]
    decreases below the predetermined threshold $\epsilon$.  The convergence of $\hat{\mu}_{\tau}$ to $\pi^{(L)}$ can no longer be measured in the KL divergence because at each iteration the measure $\hat{\mu}_{\tau}$ is no longer absolutely continuous with respect to the target $\pi^{(L)}$.  Moreover, the convergence properties as the number of particles $N \to \infty$ remains an open question.
    
    For a practical MLSVGD algorithm, an outer loop is performed over the levels $\ell=1,\ldots,L$ with the inner loop given by the SVGD updates~\eqref{eq:SVGDUpdate}.  At each intermediate level $\ell < L$ the gradients~\eqref{eq:SVGDGradient} are obtained by replacing the high-fidelity density $\pi^{(L)}$ with the low-fidelity density $\pi^{(\ell)}$.  Again, we cannot monitor the KL divergence to the target $\pi^{(\ell)}$ at each level $\ell$ as required by \eqref{eq:epsiloneq}.  Thus, the stopping criteria for when to terminate the SVGD iterations at the current level and proceed to the next level is that the norm of the gradient $\bar{g}_{\tau}$ decreases below the threshold $\epsilon$.

\subsection{Numerical results}

In the following we compare the performance of MLSVGD and SVGD.  We run both SVGD and MLSVGD with $N = 1,000$ particles, set a step size of $\delta = 0.05$, and use a Gaussian radial basis function kernel with the bandwidth parameter set to $h = 0.1$.  The bandwidth parameter is kept constant, but is comparable to the one obtained from using the median heuristic presented in \cite{NIPS2016_b3ba8f1b}.

\subsubsection{Number of iterations and runtime of SVGD and MLSVGD} Figure~\ref{fig:mlsvgd_vs_svgd} shows that with a gradient tolerance of $\epsilon = 10^{-2}$, MLSVGD achieves a speedup of a factor of five over SVGD despite requiring more iterations. Note that reducing the gradient norm below $\epsilon = 10^{-2}$ corresponds to a relative reduction of the gradient norm of more than four orders of magnitude. The results presented in Figure~\ref{fig:mlsvgd_vs_svgd} are consistent with the numerical examples presented in \cite{MLSVGD}. The runtime improvement of MLSVGD over SVGD is a result of most of the iterations being performed on the lowest fidelity model with the coarsest mesh.  MLSVGD quickly converges to the low-fidelity posterior $\pi^{(1)}$, which serves as a good initial distribution for the following two levels whereas SVGD requires many iterations at the high-fidelity level resulting in high computational costs.  The two plots in the right column of Figure~\ref{fig:mlsvgd_vs_svgd} show that both algorithms give accurate estimates of the quantity of interest in terms of the relative error 
\begin{equation}
    \mathrm{rel}({\boldsymbol \beta}) = \frac{\norm{{\boldsymbol \beta} - \hat{\boldsymbol \beta}^{\mathrm{Ref}}}_2}{\norm{\hat{\boldsymbol \beta}^{\mathrm{Ref}}}_2}\,,
\label{eq:RelErr}
\end{equation}
where $\boldsymbol \beta$ is the mean of the particles and $\hat{\boldsymbol \beta}^{\mathrm{Ref}}$ is the reference posterior mean computed with MCMC. The results suggest that the mean of the distributions of particles $\{\btheta_t^{[j]}\}_{j=1}^N$ is converging to the mean of the high-fidelity target posterior $\pi^{(L)}$.

\begin{figure}
\centering
\begin{tabular}{cc}
    \resizebox{.5\linewidth}{!}{\Large \input{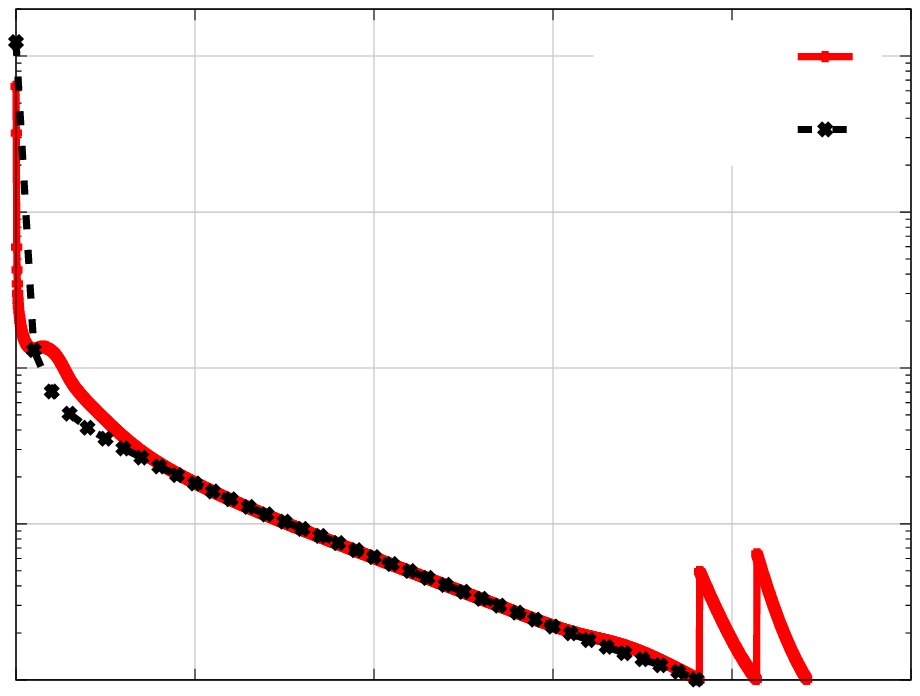}} &  \resizebox{.5\linewidth}{!}{\Large \input{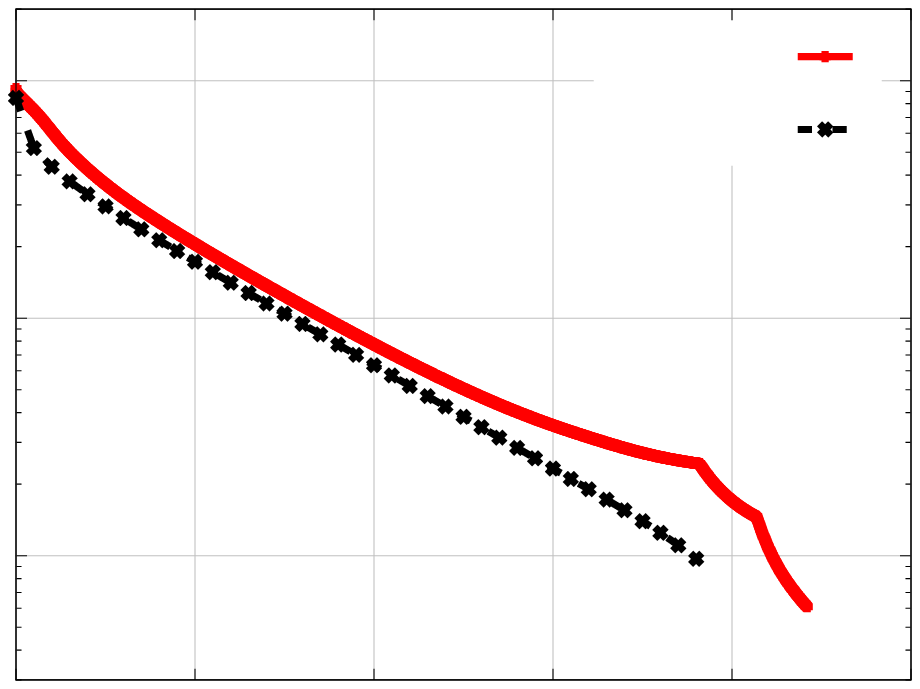}}\\
        {\small (a) gradient norm vs.~iteration} & \small (b) relative error vs. iteration \\
     \resizebox{.5\linewidth}{!}{\Large \input{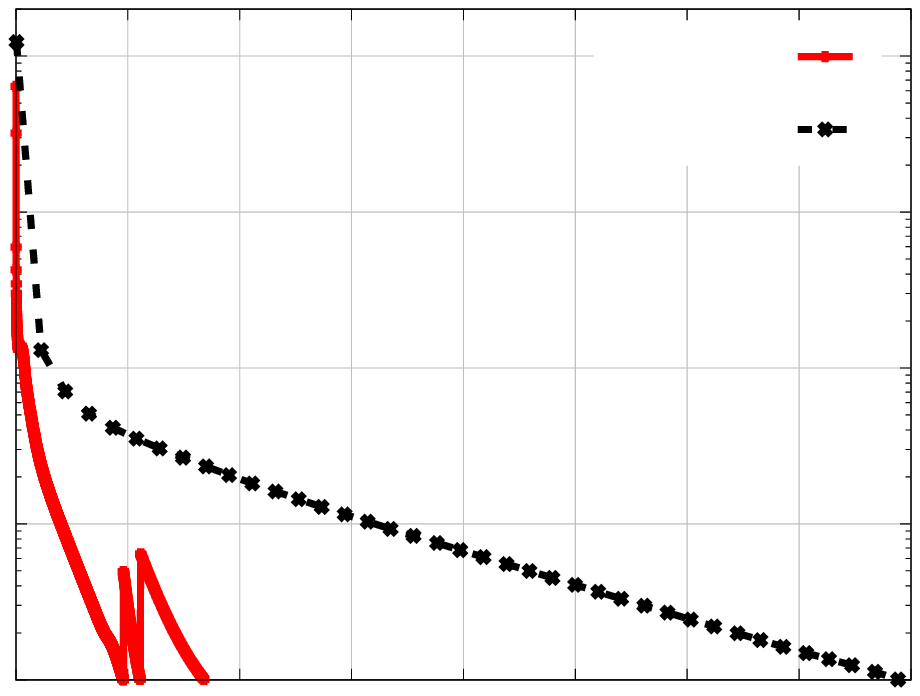}} &  \resizebox{.5\linewidth}{!}{\Large \input{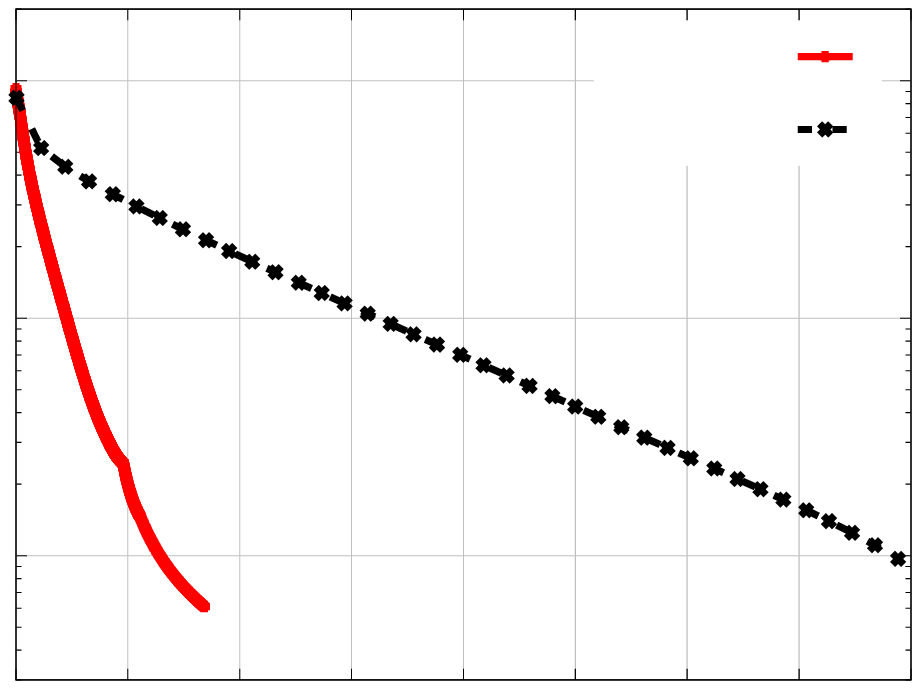}}\\
      \small (c) gradient norm vs.~runtime & \small (d) relative error vs. runtime
\end{tabular}
    \caption{{\bf (a)} The average gradient norm $\bar{g}_{\tau}$ vs. iteration for MLSVGD and SVGD with a tolerance of $\epsilon=10^{-2}$.  {\bf (b)} The relative error~\eqref{eq:RelErr} of MLSVGD and SVGD compared to an MCMC reference vs. iteration.  {\bf (c)} The average gradient norms vs. the actual runtime in hours over 32 cores.  {\bf (d)} The relative error vs. actual runtime.}
    \label{fig:mlsvgd_vs_svgd}
\end{figure}

\subsubsection{Speedups}

MLSVGD recovers a good approximation of the parameter ${\boldsymbol \beta}^*$ in less than a quarter of the time compared to SVGD because the low fidelity posteriors provide a good initialization.  Figure~\ref{fig:parameter_snapshots} compares the inferred parameter means from MLSVGD and SVGD after fixed amounts of training time.  After two hours MLSVGD has recovered the parameters whereas SVGD has not recovered them even after eight hours.  We also note that the coordinates of ${\boldsymbol \beta}$ corresponding to the right side of the domain are recovered much faster due to the location of the velocity observation points shown in Figure~\ref{fig:arolla_domain}.  Moreover, Figure~\ref{fig:forward_solution} shows the inferred velocity field $\bu$ by solving~\eqref{eq:arolla_pde} with the inferred parameter mean after approximately eight hours of run time over 32 cores.  We see that the velocity field obtained with the MLSVGD inferred parameter mean closely matches the velocity field obtained with the ground truth reference value of the mean. On the other hand, SVGD fails to recover the correct velocity field within the same amount of time.  Again we see that the left side of the domain is inaccurate due to the parameter in this region not yet being accurate.  Note that the magnitude of the velocity is overestimated for SVGD which is consistent with the fact that the parameter is underestimated since the parameter controls the frictional forces to resist the downward pull of gravity.

\begin{figure}
\vspace*{-0.35cm}\centering
\begin{tabular}{cc}
\resizebox{.48\linewidth}{!}{\Huge \input{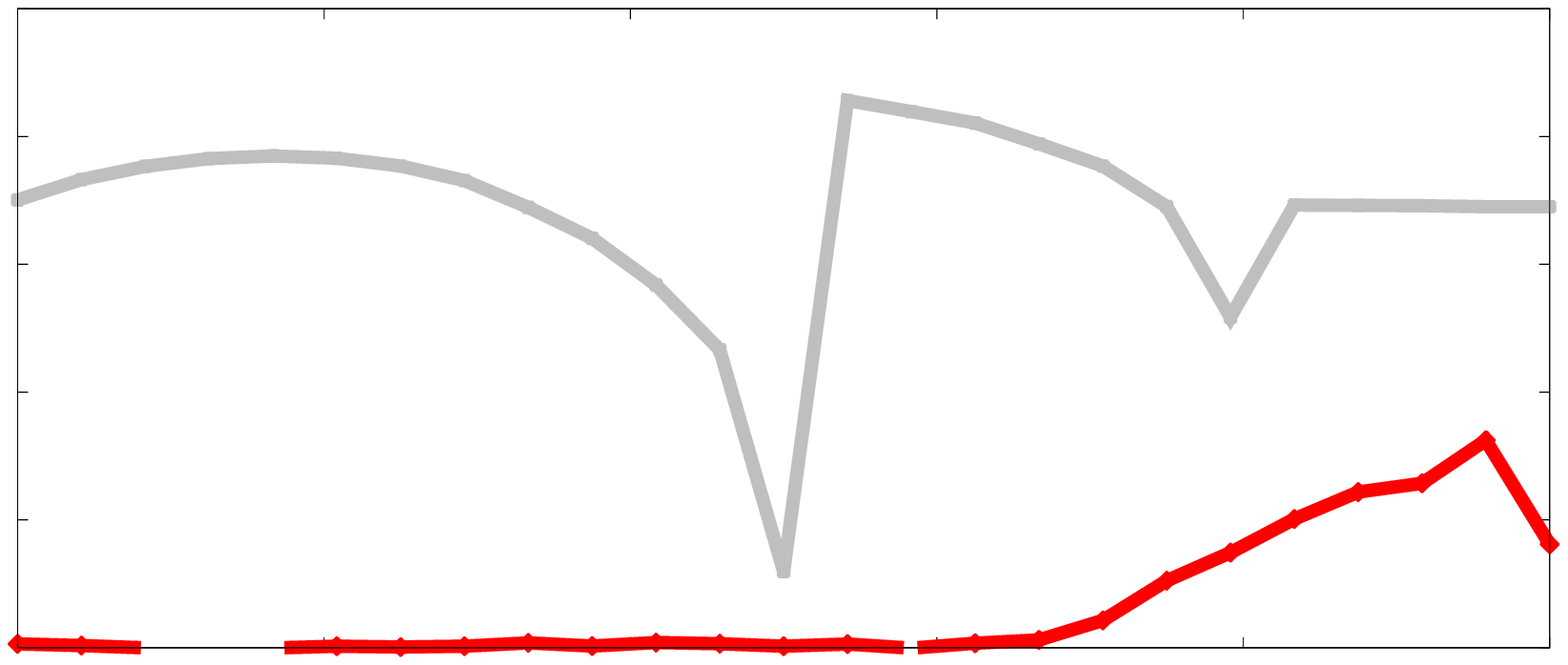}} & \hspace*{-0.2cm}
\resizebox{.48\linewidth}{!}{\Huge \input{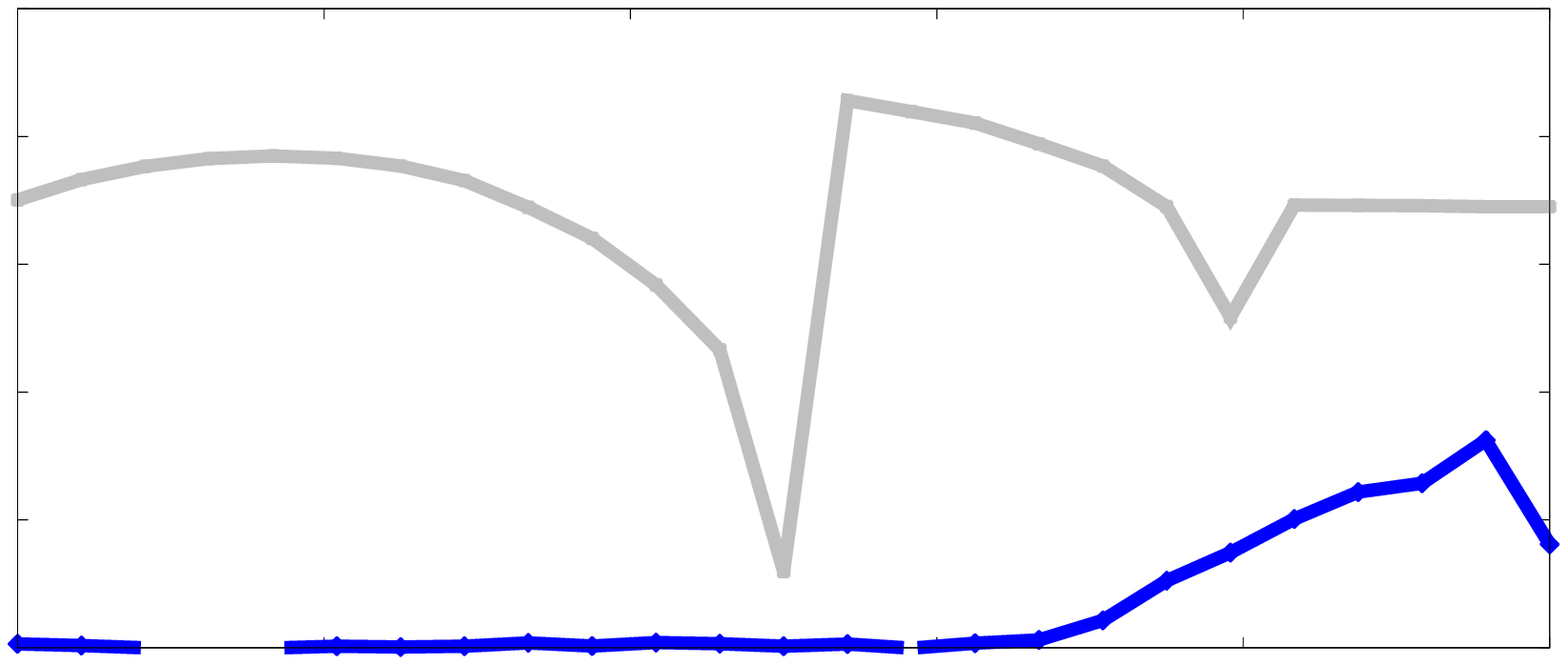}}\\
\scriptsize (a) MLSVGD, runtime $=$ 0.00 hrs. & \hspace*{-0.2cm} \scriptsize(b) SVGD, runtime $=$ 0.00 hrs.\\

\resizebox{.48\linewidth}{!}{\Huge \input{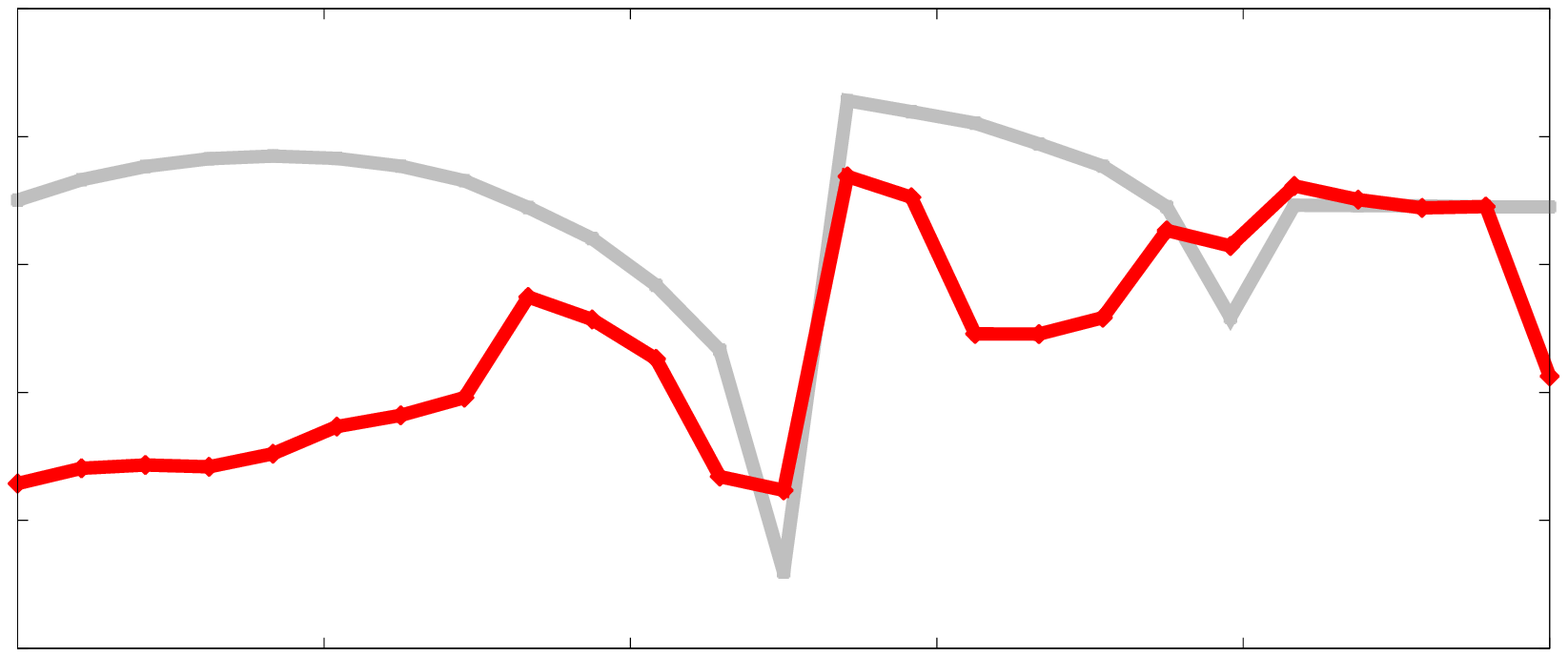}} & \hspace*{-0.2cm}
\resizebox{.48\linewidth}{!}{\Huge \input{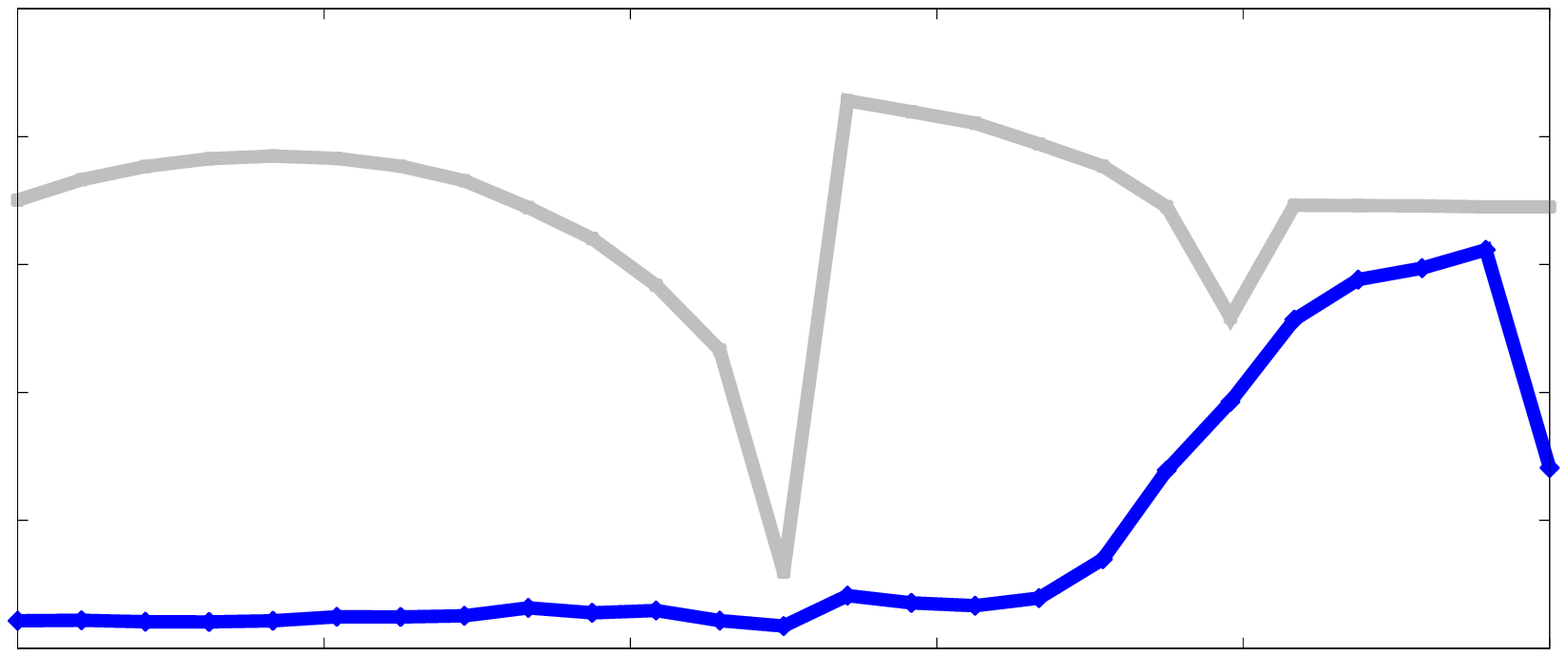}}\\
\scriptsize (c) MLSVGD, runtime $=$ 0.25 hrs.& \hspace*{-0.2cm} \scriptsize(d) SVGD, runtime $=$ 0.25 hrs.\\

\resizebox{.48\linewidth}{!}{\Huge \input{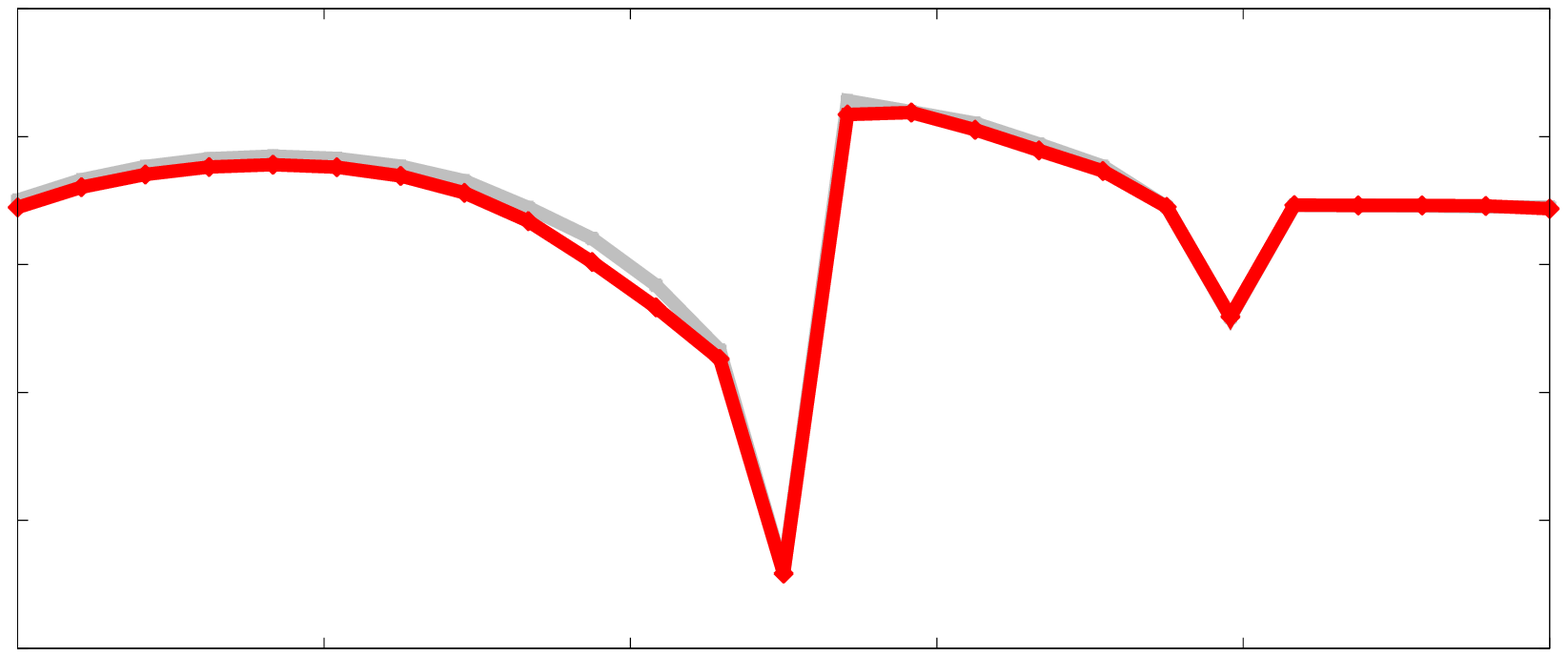}} & \hspace*{-0.2cm}
\resizebox{.48\linewidth}{!}{\Huge \input{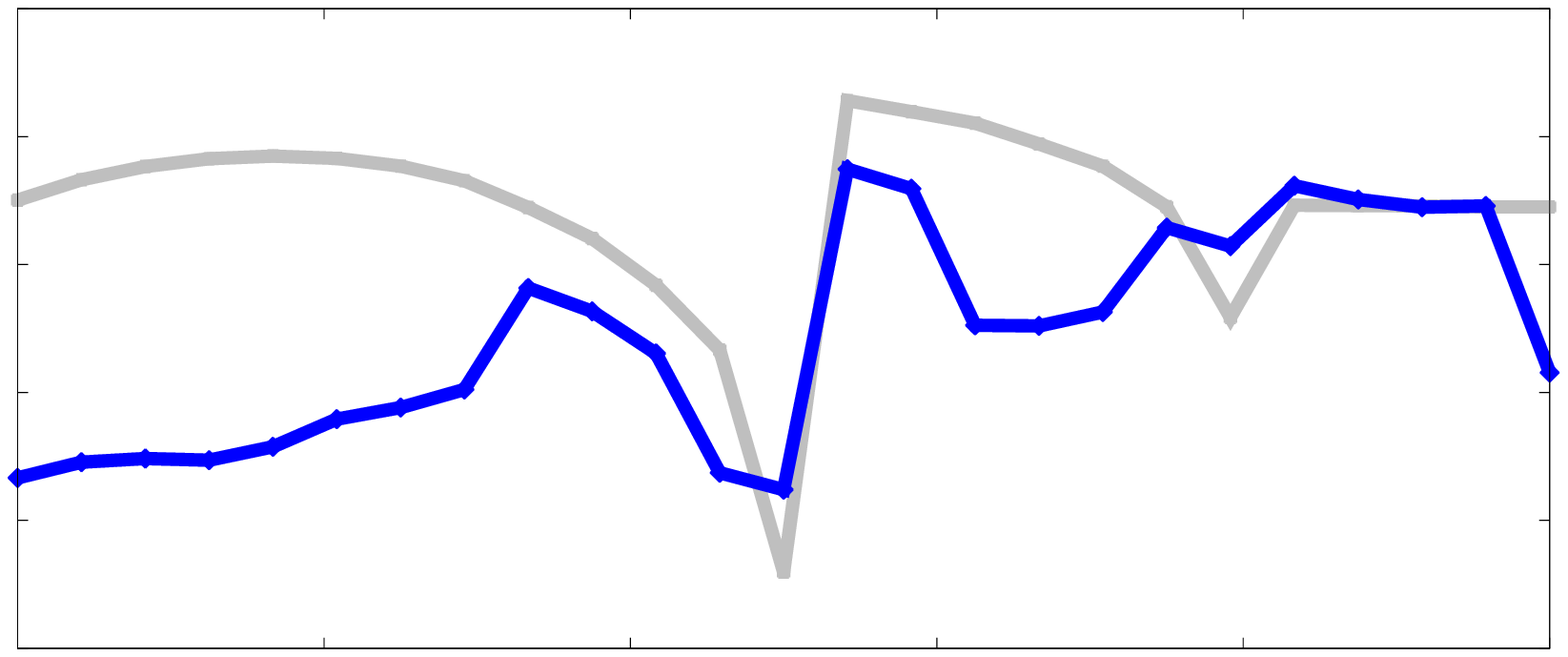}}\\
\scriptsize (e) MLSVGD, runtime $=$ 1.97 hrs. & \hspace*{-0.2cm} \scriptsize(f) SVGD, runtime $=$ 1.97 hrs.\\

\resizebox{.48\linewidth}{!}{\Huge \input{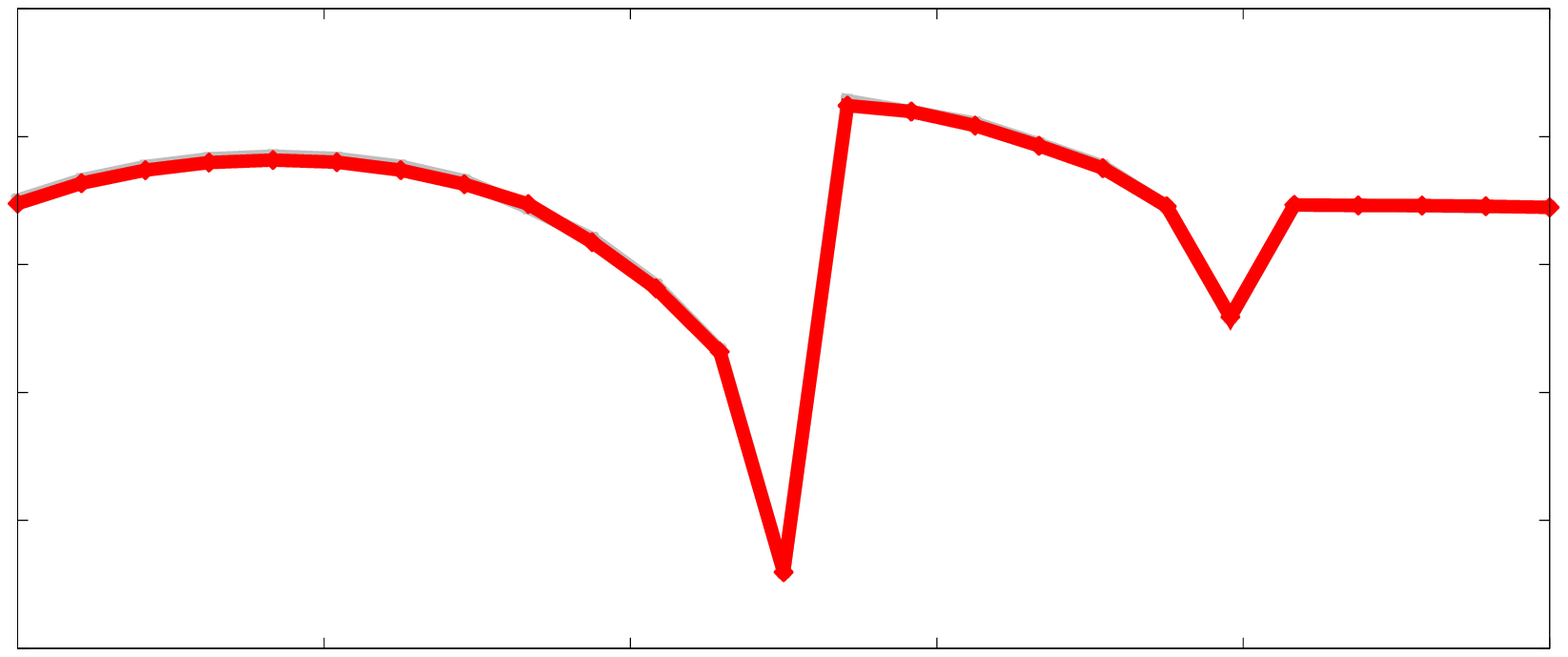}} & \hspace*{-0.2cm}
\resizebox{.48\linewidth}{!}{\Huge \input{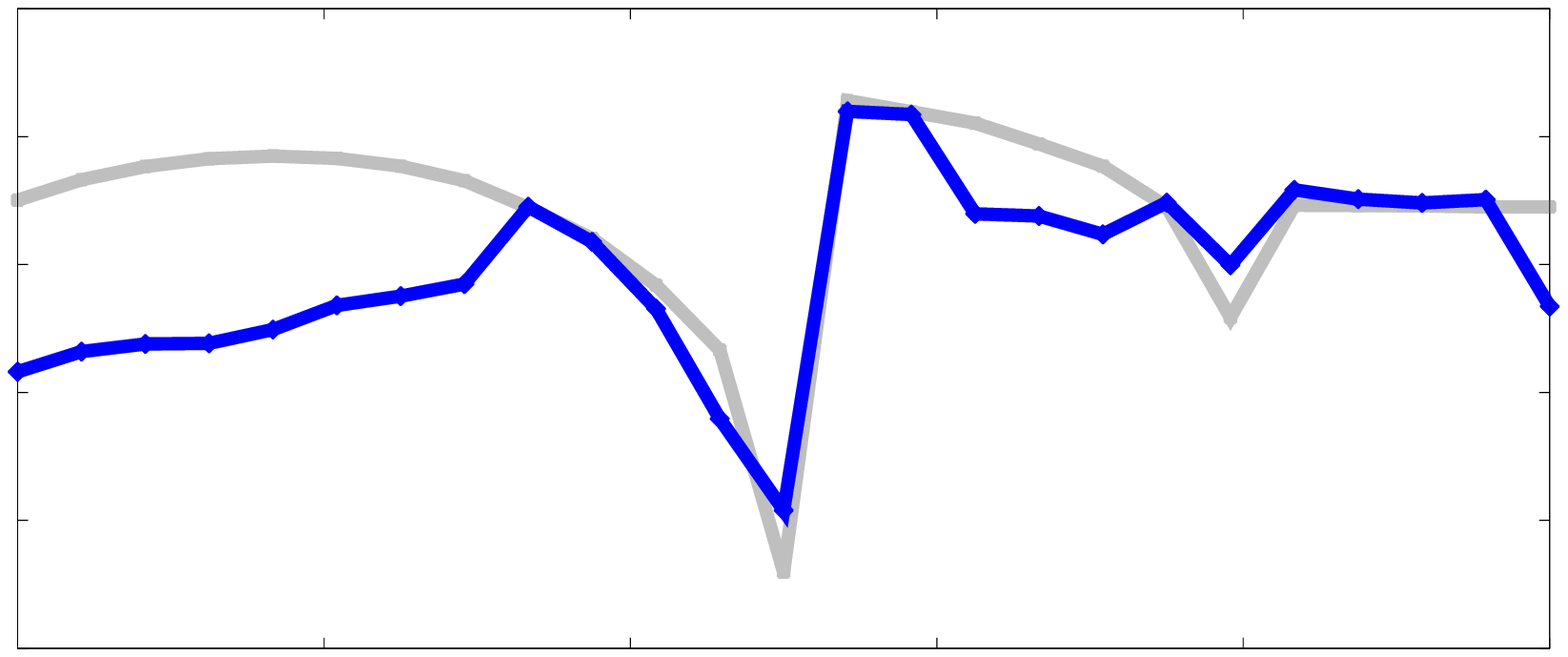}}\\
\scriptsize (g) MLSVGD, runtime $=$ 3.95 hrs. & \hspace*{-0.2cm} \scriptsize(h) SVGD, runtime $=$ 3.95 hrs.\\

\resizebox{.48\linewidth}{!}{\Huge \input{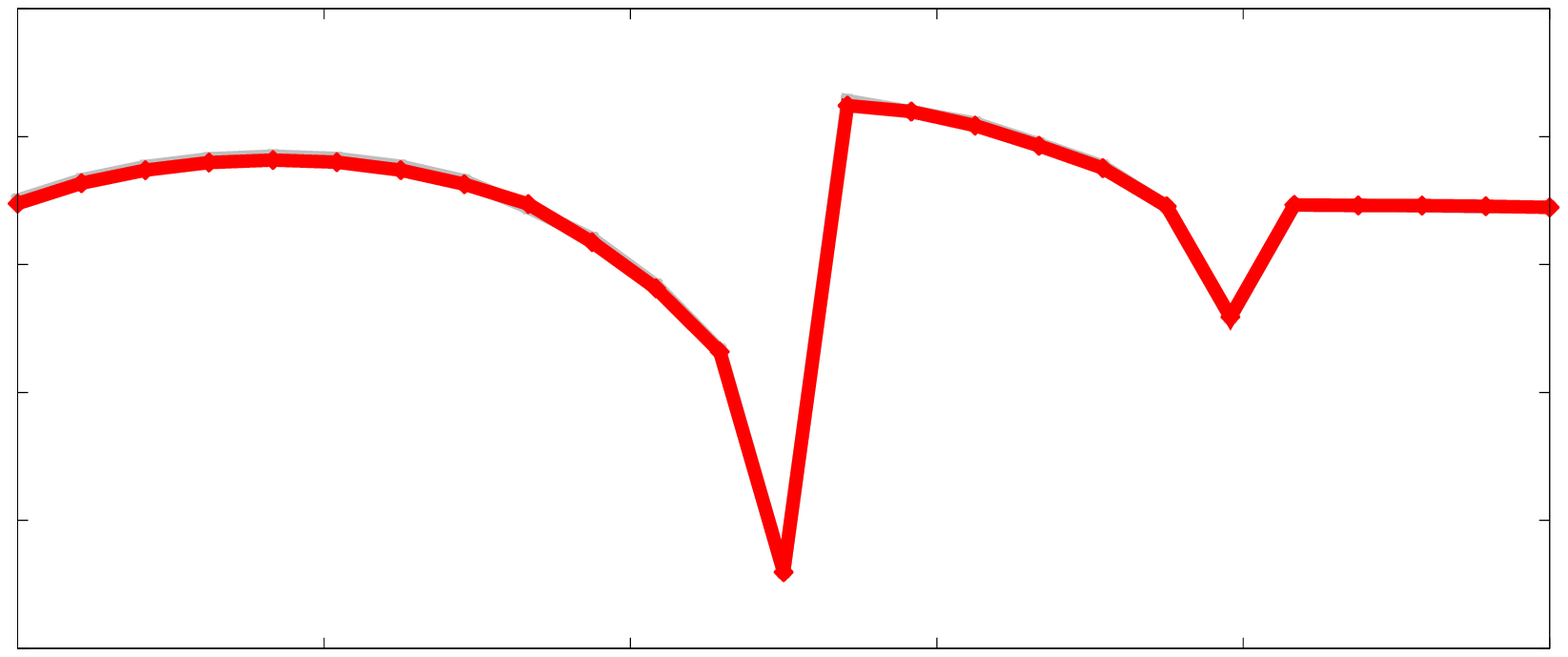}} & \hspace*{-0.2cm}
\resizebox{.48\linewidth}{!}{\Huge \input{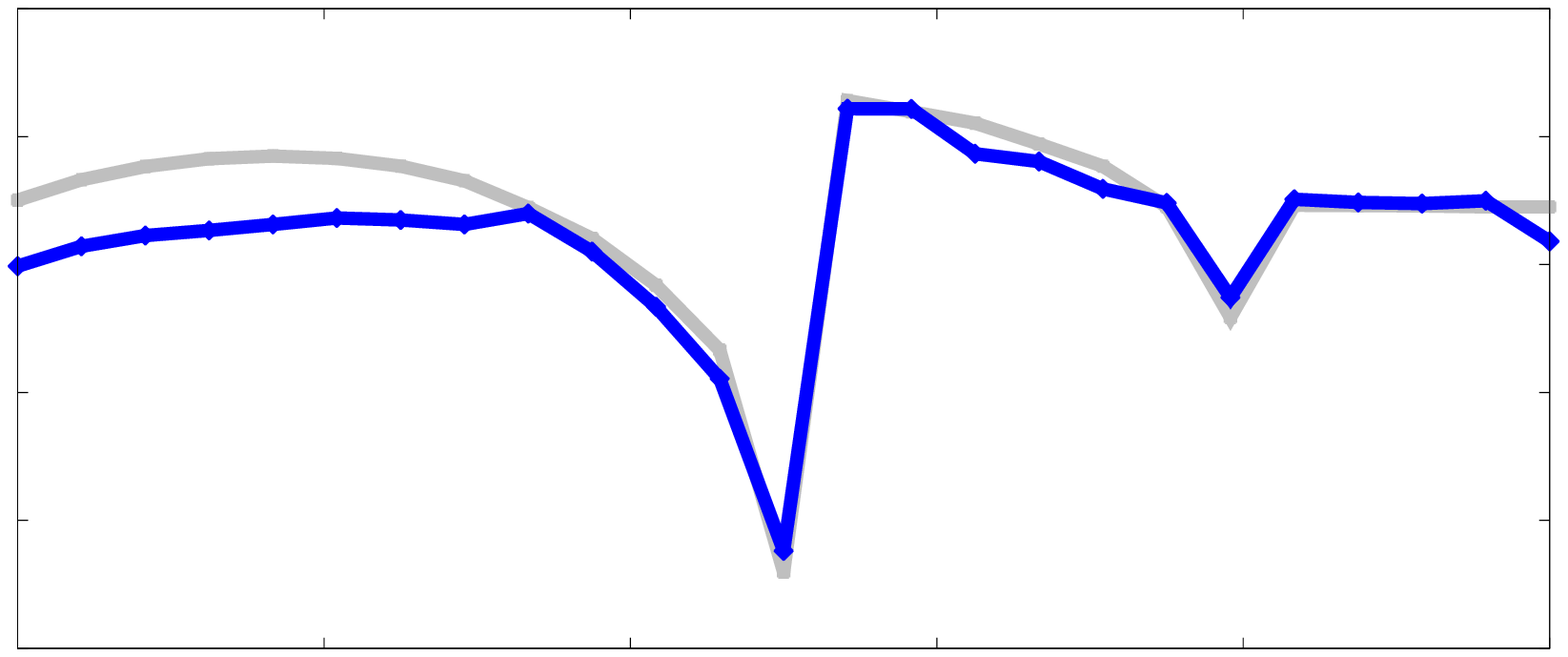}}\\
\scriptsize (i) MLSVGD, runtime $=$ 7.90 hrs. & \hspace*{-0.2cm} \scriptsize(j) SVGD, runtime $=$ 7.90 hrs.\\
\end{tabular}
\caption{({\bf Left}) Snapshots of the MLSVGD inferred parameter mean (red) at different times.  ({\bf Right}) Snapshots of the SVGD inferred parameter mean (blue) at the same times.  In each plot the solid light gray curve shows the reference value.}
\label{fig:parameter_snapshots}
\end{figure}

\begin{figure}
    \begin{tabular}{cc}
    \resizebox{.45\linewidth}{!}{
    \Large \input{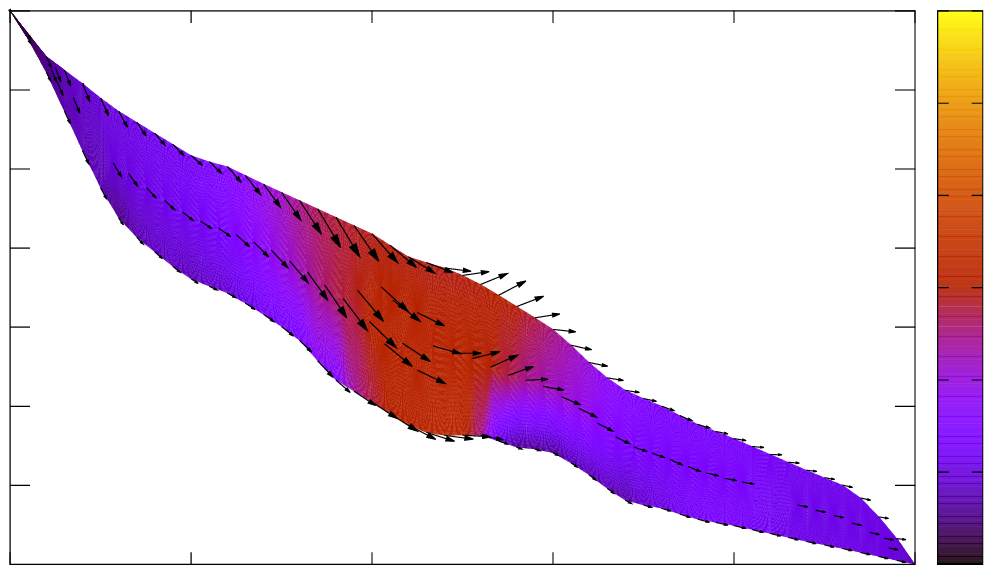}} & \hspace{0.8cm}
    \resizebox{.45\linewidth}{!}{
    \Large \input{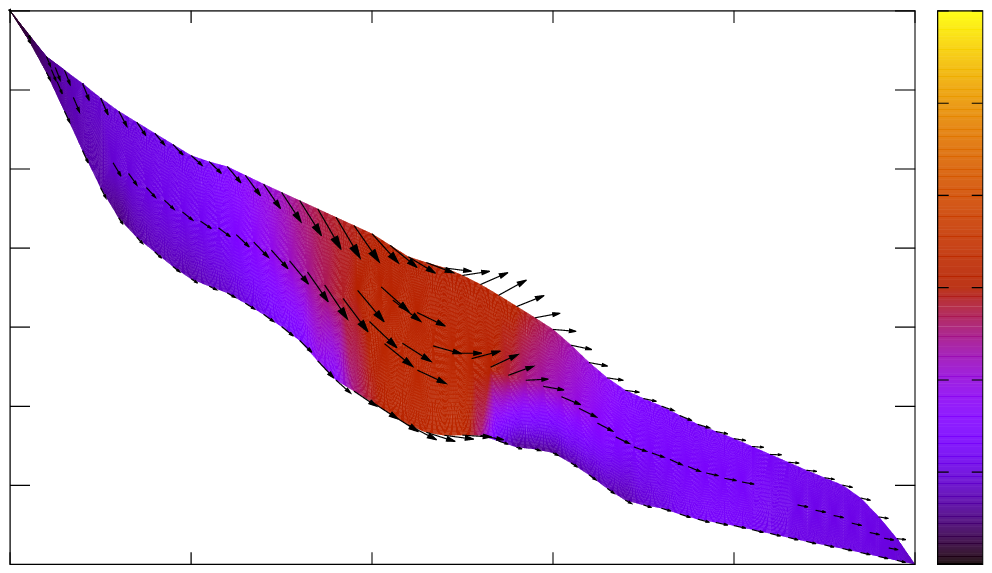}} \\
    (a) true velocity field & \hspace{0.8cm} (b) reference velocity field\\
    \resizebox{.45\linewidth}{!}{
    \Large \input{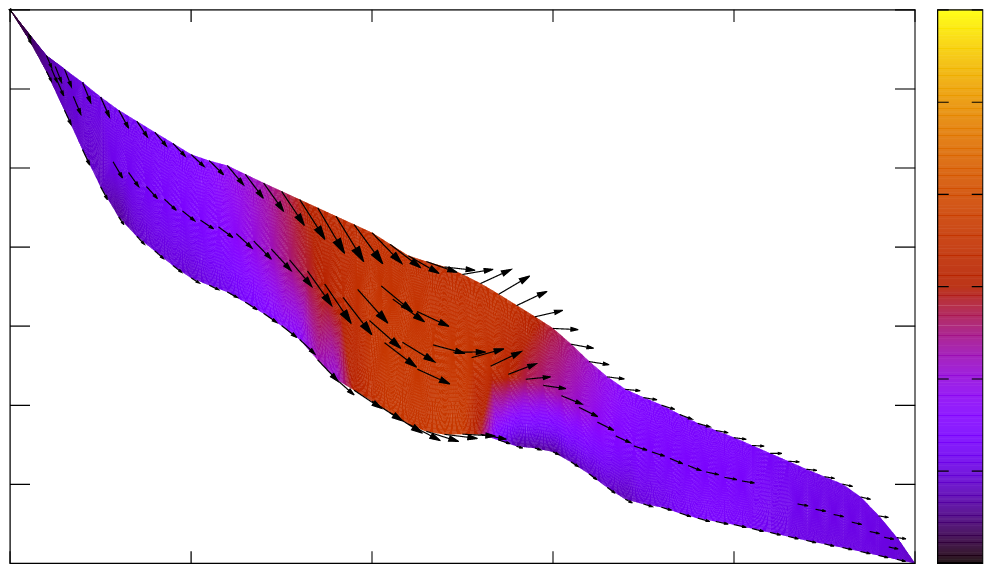}}
    & \hspace{0.8cm}
    \resizebox{.45\linewidth}{!}{
    \Large \input{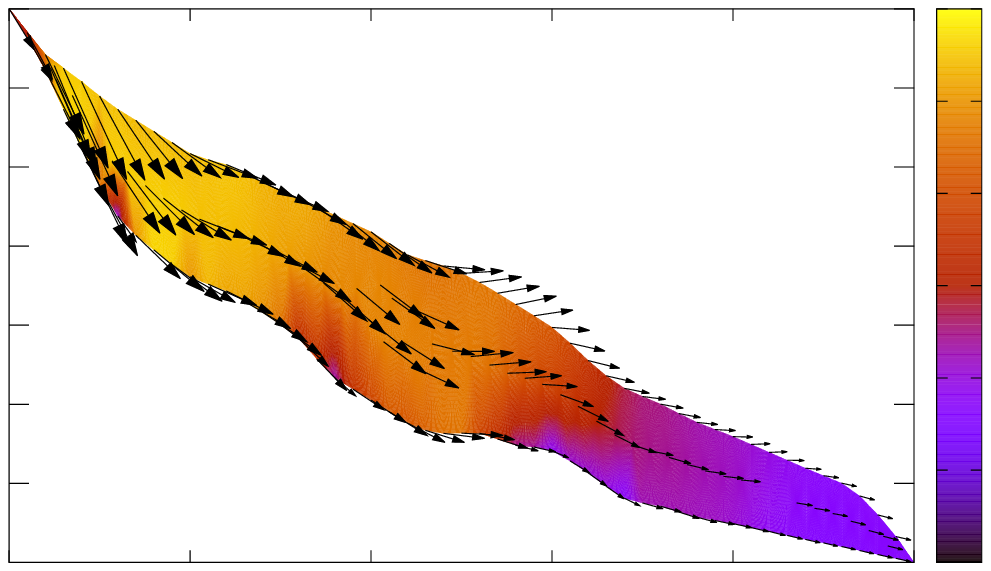}}\\
     (c) MLSVGD inferred velocity field
     & \hspace{0.8cm} (d) SVGD inferred velocity field\\
    \end{tabular}
    \caption{({\bf a}) The true velocity field given by ${\boldsymbol \beta}^*$, which is defined \eqref{eq:TrueBetaComponents}.  The color indicates the magnitude of the velocity in [$\mathrm{m}\ \mathrm{a}^{-1}$] (meters per year). ({\bf b}) The reference velocity field computed using $\hat{\boldsymbol \beta}^{\mathrm{Ref}}$ of the posterior mean.  ({\bf c}) The velocity field corresponding to the inferred parameters using MLSVGD after eight hours.  ({\bf d}) The velocity field corresponding to the inferred parameters using SVGD with equivalent costs as MLSVGD (eight hours of runtime).}
    \label{fig:forward_solution}
\end{figure}

\subsubsection{Sample quality}

Particles obtained with SVGD tend to be evenly spread out due to the repulsive interaction between particles given by the kernel. We measure sample quality with the maximum mean discrepancy (MMD)
    \[
        \mathrm{MMD}[\mu,\nu]^2 = \underset{\norm{f}_{\mathcal{H}} \le 1}{\sup} \left( \mathbb{E}_{\mu}[f] - \mathbb{E}_{\nu}[f] \right)^2 \, ,
    \]
    where $\mathcal{H}$ is the reproducing kernel Hilbert space with kernel $K$ \cite{Gretton}.  The MMD is zero if and only if the distributions $\mu = \nu$.  In practice one cannot evaluate the expectations exactly, so the following estimator \cite[Eq.~5]{Gretton} is often used instead
    \begin{multline}
    \widehat{\mathrm{MMD}}\left(\{ \bx_i \}_{i=1}^N, \{\by_j\}_{j=1}^M \right)^2    = 
     \frac{1}{N^2}\sum_{i=1}^N \sum_{i^{\prime}=1}^N K(\bx_i, \bx_{i^{\prime}}) + \frac{1}{M^2}\sum_{j=1}^M \sum_{j^{\prime}=1}^M K(\by_j, \by_{j^{\prime}})\\ - \frac{2}{NM}\sum_{i=1}^N \sum_{j=1}^M K(\bx_i, \by_{j})  \, ,
     \label{eq:mmd_estimator}
    \end{multline}
    where $\{\bx_i\}_{i=1}^N \sim \mu$ and $\{\by_j\}_{j=1}^M \sim \nu$.  To compute the MMD from the target distribution $\pi^{(L)}$ we use pCN with $\beta=0.01$ again to draw samples.  We use a burn-in period of 20,000 samples and then run 100,000 more iterations taking every 5th sample for 20,000 samples total.  These 20,000 samples serve as proxy samples from target posterior $\pi^{(L)}$.  Figure~\ref{fig:mmd} shows the estimated MMD for MLSVGD, SVGD, and MCMC.  We see that MLSVGD gives samples with comparable quality to SVGD due to the repulsive interaction between particles, and both SVGD and MLSVGD outperform MCMC (pCN) with the same sample size ($N = 1,000$).

\begin{SCfigure}
\centering
\resizebox{0.70\columnwidth}{!}{\large\input{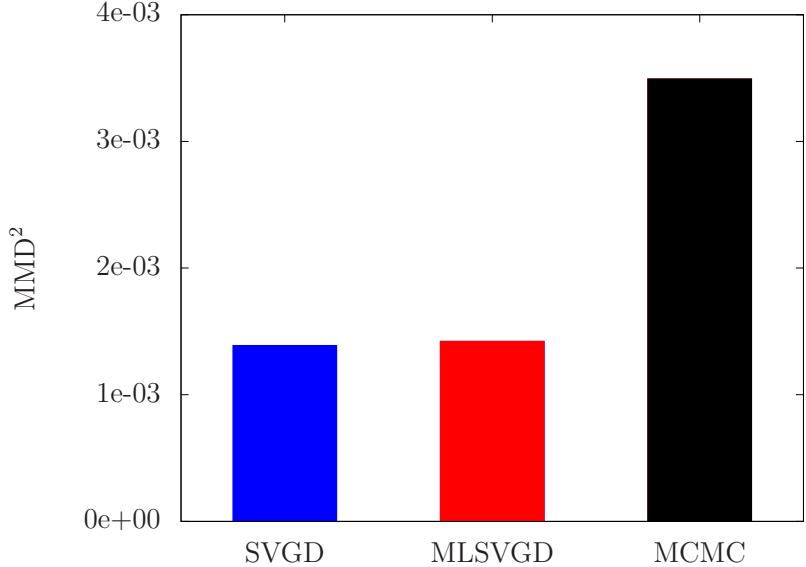}}
\caption{The estimated squared MMD using the estimator~\eqref{eq:mmd_estimator}. The MLSVGD approximation has a comparable MMD to the single-level SVGD with the high-fidelity model only.  Both have a lower MMD than MCMC suggesting higher quality samples.}
\label{fig:mmd}
\end{SCfigure}

\section{Conclusion}\label{sec:Conc}
We provided an extension of the analysis of the MLSVGD method and demonstrated with a Bayesian inverse problem of inferring a discretized basal sliding coefficient field that MLSVGD scales well to larger settings than the ones considered in prior work \cite{MLSVGD}. In particular, MLSVGD provides particles of comparable quality as SVGD but at greatly reduced computational costs in our numerical example. There are several avenues of future research. One is combining MLSVGD with the likelihood-informed projections introduced in \cite{NEURIPS2019_eea5d933}, which is especially useful in high-dimensional Bayesian inverse problems, where typically data inform only low-dimensional subspaces of the potentially high-dimensional spaces of the quantities of interest.

\section*{Declarations}

\paragraph{Funding} The first and third author acknowledge support from the Air Force Office of Scientific Research under Award Number FA9550-21-1-0222 (Dr.~Fariba Fahroo) and the National Science Foundation (NSF) under award IIS-1901091. The second and fourth author acknowledge support provided by the NSF under Grant No.~CAREER-1654311. The second author acknowledges further support provided by the NSF under Grant No.~DMS-1840265.

\paragraph{Competing interests}
The authors have no competing interests to declare that are relevant to the content of this article.

\bibliography{references}
\bibliographystyle{abbrv}

\end{document}